\documentclass[12pt]{amsart}
\usepackage[utf8]{inputenc}
\usepackage{geometry,verbatim}
\usepackage{amssymb,latexsym,amsfonts,mathtools,booktabs,stmaryrd}
\usepackage{mathrsfs} 
\usepackage{enumitem} 
\usepackage{microtype} 
\usepackage{tikz}
\usetikzlibrary{arrows,backgrounds,patterns,shapes,calc,decorations.pathreplacing}
\usepackage{todonotes}

\newenvironment{enumerate-(a)}{\begin{enumerate}[label={\upshape (\alph*)}, leftmargin=2pc]}{\end{enumerate}}
\newenvironment{enumerate-(A)}{\begin{enumerate}[label={\upshape (\Alph*)}, leftmargin=2pc]}{\end{enumerate}}
\newenvironment{enumerate-(i)}{\begin{enumerate}[label={\upshape (\roman*)}, leftmargin=2pc]}{\end{enumerate}}
\newenvironment{enumerate-(1)}{\begin{enumerate}[label={\upshape (\arabic*)}, leftmargin=2pc]}{\end{enumerate}}
\theoremstyle{plain}
\newtheorem{theorem}{Theorem}[section]

\newtheorem{question}[theorem]{Question}

\newtheorem{proposition}[theorem]{Proposition}
\newtheorem{lemma}[theorem]{Lemma}
\newtheorem{corollary}[theorem]{Corollary}
\newtheorem{claim}{Claim}[theorem]

\theoremstyle{definition}

\theoremstyle{remark}

\newtheorem*{remark*}{Remark}
\newtheorem*{remarks*}{Remarks}
\newtheorem{remarks}[theorem]{Remarks}

\newcommand{\AND}{\mathbin{\, \wedge \,}}

\renewcommand{\implies}{\Rightarrow}	%
\newcommand{\IMPLIES}{\mathbin{\, \Rightarrow \,}}
\renewcommand{\iff}{\mathbin{\Leftrightarrow }}
\newcommand{\IFF}{\mathbin{\, \Leftrightarrow \,}}
\newcommand{\FORALL}[1]{\forall {#1} \, }
\newcommand{\FORALLS}[2]{\forall\sp {#1} {#2} \, }
\newcommand{\EXISTS}[1]{\exists {#1} \, }
\newcommand{\EXISTSONE}[1]{\exists ! {#1} \, }
\newcommand{\EXISTSS}[2]{\exists\sp {#1} {#2} \, }
\newcommand{\R}{\mathbb{R}}	
\newcommand{\Q}{\mathbb{Q}}	
\newcommand{\Mid}{\boldsymbol\mid}		
\newcommand{\equalsdef}{\stackrel{\text{\tiny\rm def}}{=}}	
\newcommand{\setof}[2]{\mathopen \{{#1}\Mid{#2} \mathclose\}} 
\newcommand{\setofLR}[2]{\left \{{#1} \Mid {#2} \right\}} 
\newcommand{\set}[1]{\mathopen \{ {#1} \mathclose \}} 
\newcommand{\setLR}[1]{\left \{ {#1} \right \}} 
\newcommand{\Pow}{\mathscr{P}}		
\newcommand{\seqof}[2]{\mathopen \langle #1 \Mid #2 \mathclose \rangle} 
\newcommand{\seqofLR}[2]{\left \langle #1 \Mid #2 \right \rangle} 
\newcommand{\seq}[1]{\mathopen \langle #1 \mathclose \rangle}	
\newcommand{\seqLR}[1]{\left \langle #1 \right \rangle}	
\newcommand{\setl}[3]{\left\{{#1} \Mid {#2} \vphantom{#3} \right .}
\newcommand{\setr}[3]{\left . \vphantom{#1} \vphantom{#2}{#3}\right\}} 
\newcommand{\eq}[1]{{\boldsymbol [}{#1} {\boldsymbol ]}}	
\newcommand{\symdif}{\mathop{\triangle}}	
\newcommand{\LOC}[2]{{#1}\sb{\lfloor {#2}\rfloor}} 
\DeclareMathOperator{\PrTr}{PrTr} 
\DeclareMathOperator{\Tr}{Tr} 
\DeclareMathOperator{\IF}{IF} 
\DeclareMathOperator{\WF}{WF} 
\newcommand{\Br}[2]{{\text{\rm Br}\sp{#1}\sb{#2}}} 
\newcommand{\pre}[2]{\prescript{#1}{}{#2}}				
\DeclareMathOperator{\dom}{dom} 		
\DeclareMathOperator{\ran}{ran} 		
\newcommand{\conc}{{}\sp \smallfrown}		
\DeclareMathOperator{\lh}{lh}				
\DeclareMathOperator{\zt}{tail}				
\DeclareMathOperator{\hd}{head}				
\newcommand{\LENGTH}{\ell}
\newcommand{\flag}{{\ast}}
\newcommand{\FLAG}{\text{\sc Fl}}
\DeclareMathOperator{\Lev}{Lev}
\newcommand{\ZFC}{\mathsf{ZFC}}
\DeclareMathOperator{\Int}{Int} 		
\DeclareMathOperator{\Cl}{Cl} 		
\DeclareMathOperator{\Fr}{Fr} 		
\newcommand{\Nbhd}{{\boldsymbol N}\!} 
\newcommand{\ocinterval}[2]{\left ( {#1} ; {#2} \right ]}%
\newcommand{\cointerval}[2]{\left [ {#1} ; {#2} \right)}
\newcommand{\bSigma}{\boldsymbol{\Sigma}}
\newcommand{\bPi}{\boldsymbol{\Pi}}
\newcommand{\bGamma}{\boldsymbol{\Gamma}}

\newcommand{\bDelta}{\boldsymbol{\Delta}}
\newcommand{\dual}{\breve}	
\newcommand{\KK}{\mathbf{K}}	
\newcommand{\Gdelta}{\mathbf{G}\sb{\delta}}	
\newcommand{\MGR}{\textrm{\scshape Mgr}} 	
\newcommand{\NULL}{\textrm{\scshape Null}}	
\newcommand{\Bor}{\textrm{\scshape Bor}}	
\newcommand{\MEAS}{\text{\rm\scshape Meas}}	
\newcommand{\MALG}{\text{\rm\scshape Malg}}	
\newcommand{\Baireclass}{\mathscr{B}}
\newcommand{\Solid}{\textrm{\scshape Sld}}
\newcommand{\Dual}{\textrm{\scshape Dl}}
\newcommand{\qDual}{\textrm{\scshape qDl}}
\newcommand{\Spongy}{\textrm{\scshape Spng}}
\newcommand{\card}[1]{\mathopen | #1 \mathclose |}		
\newcommand{\leqW}{\leq\sb{\mathrm{W}}}

\newcommand{\body}[1]{\left [ {#1} \right ]} 
\newcommand{\markdef}[1]{\textbf{#1}}
\newcommand{\cchar}[1]{\boldsymbol{\chi}\sb{#1}}
\newcommand{\decode}[2]{{\boldsymbol (} {#1} {\boldsymbol )}\sb{#2}}
\newcommand{\vsection}[2]{{#1}\sb{( {#2} )}}	
\newcommand{\hsection}[2]{{#1}\sp{( {#2} )}}	
\newcommand{\density}{\mathscr{D}} 
\newcommand{\upperdensity}{\overline{\density}} 
\newcommand{\lowerdensity}{\underline{\density}} 
\newcommand{\oscillation}{\mathscr{O}} 
\DeclareMathOperator{\diam}{diam}	
\DeclareMathOperator{\Diff}{Diff}	
\DeclareMathOperator{\Blur}{Blr}
\DeclareMathOperator{\Sharp}{Shrp}
\DeclareMathOperator{\Exc}{Exc}
\newcommand{\Range}{\textsc{Rng}}
\newcommand{\SHARP}{\textsc{Shrp}}
\newcommand{\BLR}{\textsc{Blr}}
\newcommand{\appl}[2]{{#1}``{#2}}
\renewcommand{\restriction}{\mathpunct{\upharpoonright}}
\DeclareMathOperator{\proj}{p}
\newcommand{\firstreduction}{\mathcal{F}}
\newcommand{\secondreduction}{\mathcal{G}}
\newcommand{\thirdreduction}{\mathcal{H}}

\usepackage[style=alphabetic,sorting=nyt,backend=bibtex8]{biblatex}
\bibliography{density.bib}
\usepackage{hyperref}
\date{\today} 

\begin{document}
\title[Analytic sets and the density function]{Analytic sets of reals and the density function in the Cantor space}
\author{Alessandro Andretta}
\address{Dipartimento di Matematica, Università di Torino, via Carlo Alberto 10, 10123 Torino---Italy}
\email{alessandro.andretta@unito.it}
\author{Riccardo Camerlo}
\address{Dipartimento di Matematica, Politecnico di Torino, Corso Duca degli Abruzzi 24, 10129 Torino---Italy}
\email{riccardo.camerlo@polito.it}
\subjclass[2010]{03E15, 28A05}
\thanks{The authors would like to thank Vassilios Gregoriades, Alain Louveau, and John Steel for illuminating discussions.}
\begin{abstract}
We study the density function of measurable subsets of the Cantor space.
Among other things, we identify a universal set \( \mathcal{U} \) for \( \bSigma\sp{1}\sb{1} \) subsets of \( ( 0 ; 1 ) \) in terms of the density function; specifically \( \mathcal{U} \) is the set of all pairs \( ( K , r ) \) with \( K \) compact and \( r \in ( 0 ; 1 ) \) being the density of some point with respect to \( K \).
This result yields that the set of all \( K \) such that the range of their density function is \( S \cup \set{ 0 , 1 } \), for some fixed uncountable analytic set \( S \subseteq ( 0 ; 1 ) \), is \( \bPi\sp{1}\sb{2} \)-complete.
\end{abstract}
\maketitle

\section{Statement of the main results}
The density of a measurable set \( A \subseteq \pre{ \omega }{2} \) at a point \( z \in \pre{ \omega }{2} \) is \( \density\sb{A} ( z ) = \lim\sb{n \to \infty} \mu ( A \cap \Nbhd\sb{ z \restriction n } ) / \mu ( \Nbhd\sb{ z \restriction n } ) \), where \( \Nbhd\sb{s} = \setof{ x \in \pre{ \omega }{2} }{ s \subset x } \) is the basic open neighborhood determined by \( s \in \pre{ < \omega }{2} \), and \( \mu \) is the standard coin-tossing measure.
The Lebesgue density theorem says that for almost all \( z \), the value \( \density\sb A ( z ) \) is defined and it is equal to the value \( \cchar{A} ( z ) \), where \( \cchar{A} \) is the characteristic function of \( A \).
Note that \( \density\sb {A\sb 1} = \density\sb {A\sb 2} \iff A\sb 1 =\sb \mu A\sb 2 \), where \( =\sb \mu \) is the equivalence relation defined by \( A\sb 1 =\sb \mu A\sb 2 \iff \mu ( A\sb 1 \symdif A\sb 2 ) = 0 \), and hence \( \density\sb A \) depends only on the equivalence class of \( A \) in the measure algebra \( \MALG \), the collection of all Borel sets modulo \( =\sb \mu \).
The (possibly partial) function \( z \mapsto \density\sb A ( z ) \) is Borel, and \( \ran \density\sb A \) is a \( \bSigma \sp 1\sb 1 \) subset of \( [ 0 ; 1 ] \).

The main result of this paper is that the converse holds: for each analytic set \( S \subseteq ( 0 ; 1 ) \), there is a set \( A \) (which can be taken to be either closed or open) such that \( \ran \density\sb A = S \cup \set{ 0 , 1 } \) (Theorem~\ref{th:Sigma11solid}); moreover, if \( S \) is Borel we can ensure that every value in \( S \) is attained exactly once by the function \( \density\sb A \) (Theorem~\ref{th:Borelsolid}). 
Since \( \KK \), the collection of all compact subsets of \( \pre{\omega}{2} \), is a Polish space, one can try to pin-down the complexity of the families of all \( K \in \KK \) satisfying some specific property.
We show that the set of all compact sets \( K \) such that \( \ran \density\sb K = S \cup \set{ 0 , 1 } \) for some fixed analytic set \( S \), is \( \bPi\sp {1}\sb {2} \)-complete if \( S \) is uncountable, \( 2 \)-\( \bSigma\sp {1}\sb {1} \)-complete if \( S \neq \emptyset \) is countable, and \( \bPi\sp {1}\sb {1} \)-complete if \( S = \emptyset \) (Theorem~\ref{th:range}).
(Here and below \( 2 \)-\( \bGamma = \setof{ A \setminus B }{ A , B \in \bGamma } \), when \( \bGamma \) is a pointclass; alternative notations for this pointclass are \( \Diff\sb 2 \bGamma \) and \( D\sb 2 \bGamma \).)

A point \( z \in \pre{ \omega }{2} \) is blurry for \( A \) if \( \density\sb A ( z ) \) does not exist, and it is sharp if \( \density\sb A ( z ) \) exists and it is an intermediate value between \( 0 \) and \( 1 \).
A set \( A \) is solid if no point is blurry for \( A \), that is if \( \density\sb A ( z ) \) is defined for all \( z \).
(The sets \( A \) in Theorems~\ref{th:Sigma11solid} and~\ref{th:Borelsolid} can be taken to be solid.)
Suppose the value \( \density\sb A ( z ) \) is always \( 0 \) or \( 1 \) (whenever defined): if \( A \) is solid then we say that it is dualistic, otherwise we say that \( A \) is spongy.

We also prove a few results on the complexity of the families of all \( K \in \KK \) that are solid, dualistic, spongy, have a given number of sharp/blurry points, etc.
Theorem~\ref{th:hardness} shows that the set of all \( K \) that are solid (or dualistic, or that have \( n \) points that are sharp or blurry) is \( \bPi\sp {1}\sb {1} \)-complete, the set of all \( K \) whose density function attains exactly \( 1 \leq N \leq \omega \) intermediate values is \( 2 \)-\( \bSigma\sp {1}\sb {1} \)-complete, and the set of all \( K \) that are spongy is \( 2 \)-\( \bSigma\sp {1}\sb {1} \), and it is \( \bSigma\sp {1}\sb {1} \)-hard and \( \bPi\sp {1}\sb {1} \)-hard.
These results are obtained applying certain constructions from trees to compact sets, developed in Section~\ref{sec:constructions}.
These constructions are quite versatile, and could be useful elsewhere.

All these notions (being solid/spongy, having blurry/sharp points, \dots) are invariant under \( =\sb \mu \) so are well-defined in \( \MALG \).
Since \( \MALG \) is a Polish space, one can classify the corresponding sets in the context of the measure algebra.
In fact the above results for \( \KK \) imply analogous results for \( \MALG \): the set of all \( \eq{A} \) which are solid (or dualistic, or has \( n \) points that are sharp or blurry) is \( \bPi\sp {1}\sb {1} \)-complete; the set of all \( \eq{A} \) such that \( ( 0 ; 1 ) \cap \ran \density\sb A \) is a given nonempty countable set, or has size \( n \geq 1 \) is Borel \( 2 \)-\( \bSigma\sp {1}\sb {1} \)-complete; the set of all \( \eq{A} \) such that \( \ran \density\sb A = S \cup \set{ 0 , 1 } \) is Borel \( \bPi\sp {1}\sb {2} \)-complete, whenever \( S \subseteq ( 0 ; 1 ) \) is an uncountable analytic set.

Since the generic element of \( \KK \) is null, all the families of compact sets considered in this paper are meager. 
This should be contrasted with the situation in \( \MALG \), where the generic element is spongy~\cite[]{Andretta:2015kq}.

\section{Notation and preliminaries}
The notation in this paper is standard and follows closely that of~\cite[][]{Kechris:1995kc,Andretta:2013uq,Andretta:2015kq}.
We use \( \appl{f}{A} \) for the point-wise image of \( A \subseteq X \) via \( f \colon X \to Y \), that is \( \setof{ y \in Y }{ \EXISTS{x \in A} ( f ( x ) = y ) } \)---other common notations for this set such as \( f [ A ] \) or \( f ( A ) \) are not suited here, since square brackets are used for the body of a tree, and round brackets could be ambiguous. 
For the effective aspects of descriptive set theory the standard reference is~\cite[][]{Moschovakis:2009fk}.
We also introduce a few technical tools that will come handy.

\subsection{Sequences and trees}\label{subsec:sequences&trees}
\subsubsection{Sequences}
Fix a nonempty set \( I \).
The \markdef{length} of \( x \in \pre{ \leq \omega }{I} \) is the ordinal \( \lh ( x ) = \dom ( x ) \).
The \markdef{concatenation of \( s \in \pre{ < \omega }{I} \) with \( x \in \pre{ \leq \omega }{I} \)} is denoted by \( s \conc x \) and belongs to \( \pre{ \leq \omega }{ I } \).
We will often blur the difference between the sequence \( \seqLR{ i } \) of length \( 1 \) with its unique element \( i \) and write \( t \conc i \) instead of \( t \conc \seq{ i } \).
The sequence of length \( N \leq \omega \) that attains only the value \( i \) is denoted by \( i\sp { ( N ) } \).
If \( A \subseteq \pre{ \omega }{I} \) and \( s \in \pre{ < \omega }{I} \) let \( s \conc A = \setof{ s \conc x }{ x \in A} \subseteq \pre{\omega}{I} \). 

For \( x , y \in \pre{ \omega }{I} \), the element \( x \oplus y \in \pre{ \omega }{I} \) is defined by
\begin{equation}\label{eq:oplus}
( x \oplus y ) ( n ) = 
	\begin{cases}
	x ( k ) & \text{if } n = 2k ,
	\\
	y ( k ) & \text{if } n = 2k+ 1 .
	\end{cases}
\end{equation}
Equation~\eqref{eq:oplus} above defines an operation even when \( x , y \) are finite sequences of the same length \( N < \omega \), so that \( x \oplus y = \seq{ x ( 0 ) , y ( 0 ) , \dots , x ( N - 1 ) , y ( N - 1 ) } \).

Fix a recursive bijection \( J \colon \omega \times \omega \to \omega \). 
Then any element of \( x \in \pre{ \omega }{I} \) encodes an \( \omega \)-sequence \( \seqof{ \decode{ x }{n} }{ n \in \omega } \) of elements of \( \pre{ \omega }{I} \), 
\begin{equation}\label{eq:decode}
 \decode{ x }{n} \colon \omega \to I , \qquad \decode{ x }{n} ( m ) = x ( J ( n , m ) ) .
\end{equation}

For \( s \in \pre{ < \omega }{ 2} \), the longest initial segment of \( s \) ending with a \( 1 \) is
\begin{equation}\label{eq:head}
\hd ( s ) = \begin{cases}
s \restriction n + 1 & n \text{ is the largest \( k \) such that \( s ( k ) = 1 \), if it exists,}
\\
\emptyset & \text{otherwise,}
\end{cases} 
\end{equation}
and the length of the final segment of \( s \) ending with \( 0 \)s is 
\begin{equation}\label{eq:tail}
 \zt ( s ) = \text{the unique \( k \) such that } \bigl ( s = \hd ( s ) \conc 0 \sp { ( k )} \bigr ). 
\end{equation}
The map \( \pre{ < \omega }{ \omega } \to \pre{ < \omega }{2} \), \( t \mapsto \check{t} \) defined by
\begin{equation}\label{eq:check(t)}
 t = \seqLR{ t ( 0 ) , \dots , t ( n ) } \mapsto \check{t} = 0\sp {( t ( 0 ) )} \conc 1 \conc 0\sp {( t ( 1 ) )} \conc 1 \conc \dots \conc 0\sp {( t ( n ) )} \conc 1 
 \end{equation}
is injective and admits a left inverse, 
\begin{equation}\label{eq:hat(t)}
\pre{ < \omega }{ 2} \to \pre{ < \omega }{ \omega } , \quad s \mapsto \hat{ s}
\end{equation}
defined by 
\[
 \hat{ s} = \widehat{ \hd ( s ) } = \text{the unique \( t \) such that } \bigl ( \check{t} = \hd ( s ) \bigr ) .
\]
Note that the range of \( t \mapsto \check{ t} \) is \( \setofLR{ s \in \pre{ < \omega }{ 2} }{ \zt ( s ) = 0 } = \setofLR{ s \in \pre{ < \omega }{ 2} }{ \hd ( s ) = s } \).

For any \( s \in \pre{ < \omega }{2} \), the number of \( 1 \)s in \( s \) is 
\begin{equation}\label{eq:virtuallength}
\LENGTH ( s ) = \card{\setofLR{i < \lh s }{ s ( i ) = 1 } } = \lh ( \hat{s} ) .
\end{equation}
Moreover \( s \) is said to be \markdef{even} or \markdef{odd} depending on the parity of \( \LENGTH ( s ) \).

\subsubsection{Trees}
Let \( I \) be a nonempty set.
The \markdef{downward closure} of \( X \subseteq \pre{ < \omega }{I} \) is 
\[ 
{\downarrow} X \equalsdef \setof{ s \restriction n }{ s \in X \wedge n \in \omega } .
\]
A \markdef{tree} on \( I \) is \( {\downarrow}X \) of some set \( \emptyset \neq X \subseteq \pre{ < \omega }{I} \), and \( \Tr\sb I \) is the set of all trees on \( I \). 
In other words, an element of \( \Tr\sb I \) is a nonempty collection of finite sequences from \( I \), closed under initial segments.
A tree on \( I \times J \) is construed as a set of pairs \( ( u , v ) \) with \( u \in \pre{ < \omega }{I} \), \( v \in \pre{ < \omega }{J} \), and \( \lh u = \lh v \).
If \( T \in \Tr\sb { I \times J } \) and \( y \in \pre{ \omega }{J} \), then
\[
T ( y ) = \setof{ t \in \pre{ < \omega }{I} }{ ( t , y \restriction \lh t ) \in T }
\]
is a tree on \( I \).
A tree \( T \) is \markdef{pruned} if \( \FORALL{t \in T} \EXISTS{s \in T} ( t \subset s ) \) and \( \PrTr \sb I \) is the set of all pruned trees on \( I \).
If \( 0 < \card{ I } \leq \omega \) then \( \card{ \pre{ < \omega }{I} } = \omega \), and therefore \( \Tr\sb I\) and \( \PrTr\sb I\) can be coded as subsets of \( \Pow ( \omega ) \), so with a minor abuse of notation we construe them as subsets of the Cantor space; in fact as subsets of \( \pre{\omega}{2} \), \( \Tr\sb I \) is closed and \( \PrTr\sb I \) is \( \Gdelta \).
For the sake of simplicity, when \( I = \omega \) we write \( \Tr \) and \( \PrTr \) instead of \( \Tr\sb \omega \) and \( \PrTr\sb \omega \).
A tree \( T \) on \( \omega \) is \markdef{gapless} if \( \FORALL{t \in T} \FORALL{n , m }( t \conc \seq{ n } \in T \wedge m \leq n \implies t \conc \seq{ m } \in T ) \).

The \markdef{body} of \( T \in \Tr\sb I\) is \( \body{T} = \setofLR{ x \in \pre{ \omega }{I} }{ \FORALL{n \in \omega } ( x \restriction n \in T ) } \), and its elements are called \markdef{branches} of \( T \).
Thus \( \pre{ \omega }{I} = \body{ \pre{ < \omega }{ I } } \).
The set \( \body{T} \) is a topological space with the topology generated by the sets 
\[ 
\Nbhd\sb t \sp {\body{T}}= \Nbhd\sb t = \setofLR{x \in \body{T} }{ x \supseteq t }
\] 
with \( t \in T \).
This topology is induced by the complete metric 
\[ 
d\sb T ( x , y ) = 
	\begin{cases}
	2\sp {-n} & \text{if \( n \) is least such that } x ( n ) \neq y ( n ) , 
	\\ 0 & \text{if } x = y .
	\end{cases}
\]
Every nonempty closed subset of \( \pre{ \omega }{I} \) is of the form \( \body{T} \) for some unique \( T \in \PrTr\sb I \).

A function \( \varphi \colon T \to S \) between pruned trees is
\begin{itemize}
\item
\markdef{monotone} if \( t\sb 1 \subseteq t\sb 2 \implies \varphi ( t\sb 1 ) \subseteq \varphi ( t\sb 2 ) \),
\item
\markdef{continuous} if it is monotone and \( \lim\sb n \lh \varphi ( x \restriction n ) = \infty \) for all \( x \in \body{T} \). 
\end{itemize}
A continuous \( \varphi \) induces a continuous function
\[ 
f \sb \varphi \colon \body{T} \to \body{S} , \quad f ( x ) = \textstyle\bigcup\sb {n} \varphi ( x \restriction n ) ,
\] 
and every continuous function \( \body{T} \to \body{S} \) arises this way.
(In Section~\ref{subsubsec:continuousfunctionsfromBaire} this construction will be extended to continuous functions from \( \body{T} \) to \( [ 0 ; 1 ] \).)

The \markdef{localization} of \( X \subseteq \pre{ \leq \omega }{I} \) at \( s \in \pre{ < \omega }{I} \) is
\[
\LOC{X}{s} = \setofLR{ t \in \pre{ \leq \omega }{I } }{ s \conc t \in X } . 
\]
Thus if \( A \subseteq \pre{\omega}{I} \) then \( s \conc \LOC{A}{s} = A \cap \Nbhd\sp {\mathcal{X}}\sb s \), where \( \mathcal{X} = \body{ \pre{ < \omega }{ I } } = \pre{ \omega }{ I } \).
Note that if \( T \) is a tree on \( I \) and \( t \in T \), then \( \body{\LOC{T}{t}} = \LOC{ \body{T}}{t} \).

\subsection{Polish spaces}
A topological space is \markdef{Polish} if it is separable and completely metrizable.
If \( X \) is Polish, then so is \( \KK ( X ) \) the hyperspace of all compact subsets of \( X \) with the Vietoris topology.
In this paper \( \KK ( \pre{\omega}{2} ) \) will simply be denoted by \( \KK \).
If \( T \) is a tree on a countable set \( I \), then \( \body{T} \) is a Polish space, as witnessed by the metric \( d\sb T \) defined above. 
A function \( f \colon X \to Y \) between metrizable spaces is of \markdef{Baire-class} \( \alpha \) if the preimage of an open set is in \( \bSigma\sp {0}\sb { \alpha + 1} \).
The set of all functions of Baire-class \( \alpha \) is denoted by \( \Baireclass\sb \alpha ( X , Y ) \) or simply \( \Baireclass\sb \alpha \) when \( X \) and \( Y \) are clear from the context.

\subsubsection{The Cantor and Baire spaces}
The \markdef{Cantor space} \( \pre{ \omega }{2} \) is the body of \( \pre{ < \omega }{2} \). 
It is a compact space, and it is homeomorphic to \( \body{T} \) for any perfect pruned tree \( T \) which is finitely branching, that is \( \setof{u \in T}{ t \subset u \wedge \lh ( u ) = \lh ( t ) + 1 } \) is finite, for all \( t \in T \).

The \markdef{Baire space} \( \pre{ \omega }{ \omega } \) is the body of \( \pre{ < \omega }{ \omega } \).
It is homeomorphic to \( [ 0 ; 1 ] \setminus \Q \), and all of its subsets that are countable union of compact sets have empty interior, and hence it is not homeomorphic to \( \pre{\omega}{2} \).

The map \( t \mapsto \check{t} \) induces a continuous function 
\begin{equation}\label{eq:BaireinCantor0}
 \boldsymbol{h} \colon \pre{ \omega }{ \omega } \to \pre{ \omega }{2} , \quad x \mapsto 0\sp {( x ( 0 ) )} \conc 1 \conc 0\sp {( x ( 1 ) )} \conc 1 \conc \dots .
\end{equation}
The range of \( \boldsymbol{h} \) is
\[
\mathcal{N} = \setofLR{y \in \pre{ \omega }{2} }{ \EXISTSS{\infty}{n} y ( n ) = 1 } ,
\]
and \( \boldsymbol{h} \colon \pre{ \omega }{ \omega } \to \mathcal{N} \) is a homeomorphism.
Note that \( \mathcal{N} \) is a \( \Gdelta \) subset of the Cantor space, and it is the set of all \( x \in \pre{ \omega }{2} \) such that \( x \restriction n \) changes from even to odd infinitely often.
The function \( \hd \) can be defined on \( \pre{ \omega }{2} \setminus \mathcal{N} \) by letting 
\[ 
\hd ( x ) = \text{the shortest \( s \in \pre{ < \omega }{2} \) such that } \bigl ( s \conc 0 \sp {( \omega )} = x \bigr ) .
\]
An \( x \in \pre{ \omega }{2} \setminus \mathcal{N} \) is said to be \markdef{even} or \markdef{odd} iff \( \hd ( x ) \) is even or odd, that is if the number of \( 1 \)s in the sequence \( x \) is even or odd. 

\subsubsection{Approximation of real valued continuous functions}\label{subsubsec:continuousfunctionsfromBaire}
Given \( T \in \PrTr \), a function \( f \colon \body{T} \to [ 0 ; 1 ] \) is \markdef{Lipschitz} iff \( \FORALL{ t \in T } ( \diam ( \appl{f}{ \Nbhd\sb t } ) \leq 2\sp { - \lh t } ) \); equivalently, iff \( \card{ f ( x ) - f (y ) } \leq d\sb T ( x , y ) \) for all \( x , y \in \body{T} \).
Our definition of ``Lipschitz function'' agrees with the usual definition of ``Lipschitz function with constant \( \leq 1 \)'' in the sense of metric spaces.

\begin{lemma}\label{lem:tight}
If \( f \colon \body{T} \to [ 0 ; 1 ] \) is continuous, then there is a gapless pruned tree \( U \) on \( \omega \) and a homeomorphism \( h \colon \body{U} \to \body{T} \) such that \( f \circ h \colon \body{U} \to [ 0 ; 1 ] \) is Lipschitz.
\end{lemma}

\begin{proof}
Let \( A\sb 0 = \set{ \emptyset} \) and \( A\sb {n + 1} \) the set of all minimal \( t \in T \) properly extending some node in \( A\sb n \) such that \( \diam ( \appl{ f }{\Nbhd\sb t} ) \leq 2\sp {- n - 1 } \).
By continuity of \( f \), for every \( x \in \body{T} \) there is a unique \( m \) such that \( x \restriction m \in A\sb n \).
The set \( S = \bigcup\sb {n} A\sb n \) ordered under inclusion is a countable rooted combinatorial tree without terminal nodes, and hence it is isomorphic to a pruned gapless tree \( U \) on \( \omega \).
The isomorphism \( U \to S \) induces a homeomorphism \( h \colon \body{U} \to \body{T} \) and by construction \( f \circ h \) is Lipschitz.
\end{proof}

Given \( T \in \PrTr \), a continuous \( f \colon \body{T} \to [ 0 ; 1 ] \), a countable \( Q \subseteq ( 0 ; 1 ) \) such that \( \ran f \subseteq \Cl Q \), a \markdef{\( Q \)-approximation} of \( f \) is a map \( \varphi \colon T \to Q \) such that \( f ( x ) = \lim\sb n \varphi ( x \restriction n ) \) for all \( x \in \body{T} \). 
A continuous \( f \colon \mathcal{N} \to [ 0 ; 1 ] \) can be identified with a continuous \( \bar{f} \colon \pre{\omega}{\omega} \to [ 0 ; 1 ] \), and a \( Q \)-approximation of \( f \) is a map
\[
 \varphi \colon \setof{ s \in \pre{ < \omega}{2} }{ \zt ( s ) = 0 } \to Q 
\]
such that \( \bar{ \varphi } \colon \pre{ < \omega }{ \omega } \to Q \), \( \bar{ \varphi } ( s ) = \varphi ( \check{ s } ) \) with \( \check{ s} \) as in~\eqref{eq:check(t)}, is a \( Q \)-approximation of \( \bar{f} \).
The set 
\[ 
\mathbb{D} = \setofLR{ k \cdot 2\sp {- n}}{ 0 < k < 2\sp n \wedge n \in \omega } 
\] 
of all \markdef{dyadic rationals} of \( ( 0 ; 1 ) \) is countable and dense in \( [ 0 ; 1 ] \), so for any continuous \( f \colon \body{T} \to [ 0 ; 1 ] \) we can compute a \( \mathbb{D} \)-approximation, often called a \markdef{dyadic approximation}.
If we fix a well-ordering \( \lhd \) of \( \mathbb{D} \) of order type \( \omega \) we can choose a \markdef{canonical dyadic approximation} \( \varphi \) of \( f \colon \body{T} \to [ 0 ; 1 ] \) by requiring that
\[
 \varphi ( t ) = \text{the \( \lhd \)-minimum of } \mathbb{D} \cap I\sb t
\]
where
\[
I\sb t = 
	\begin{cases}
	\left ( \inf ( \appl{f}{ \Nbhd\sb t } ) ; \sup ( \appl{f}{ \Nbhd\sb t } ) \right ) & \text{if } \inf ( \appl{f}{ \Nbhd\sb t } ) < \sup ( \appl{f}{ \Nbhd\sb t } )
	\\
	\left ( \inf ( \appl{f}{ \Nbhd\sb t } ) - 2\sp {- \lh ( t ) - 1 } ; \sup ( \appl{f}{ \Nbhd\sb t } ) + 2\sp {- \lh ( t ) - 1 }\right ) & \text{if } \inf ( \appl{f}{ \Nbhd\sb t } ) = \sup ( \appl{f}{ \Nbhd\sb t } ) .
	\end{cases}
\]
Thus if \( \varphi \colon T \to \mathbb{D} \) is the canonical dyadic approximation of a Lipschitz \( c \colon \body{T} \to [ 0 ; 1 ] \), then
\begin{equation}\label{eq:errorindyadicapproximation}
t \conc \seq{ h } , t \conc \seq{ k } \in T \IMPLIES \card{ \varphi ( t \conc \seq{ h } ) - \varphi ( t \conc \seq{ k } ) } < 2\sp { - \lh t } .
\end{equation}
Note that \( \mathbb{D} \) is the set of all values \( \mu ( D ) \) with \( \emptyset \subset D \subset \pre{ \omega }{2}\) clopen.
For each \( r \in \mathbb{D} \) we pick a clopen set \( D \) such that \( \mu ( D ) = r \), so whenever we say ``choose \( D \) such that \( \mu ( D ) \) has value \( r \in \mathbb{D} \)'' it is understood that such choice is canonical.

\subsection{Borel and projective sets}
\subsubsection{Pointclasses}
Let \( X \) be a Polish space.
As customary in descriptive set theory, we use \( \bSigma\sp {0}\sb { \alpha } ( X ) \) and \( \bPi\sp {0}\sb { \alpha } ( X ) \) to denote the Baire additive and multiplicative pointclasses, and set \( \bDelta\sp {0}\sb { \alpha } ( X ) = \bSigma\sp {0}\sb { \alpha } ( X ) \cap \bPi\sp {0}\sb { \alpha } ( X ) \).
In particular \( \bSigma\sp {0}\sb {1} ( X ) \), \( \bPi\sp {0}\sb {1} ( X ) \), and \( \bDelta\sp {0}\sb {1} ( X ) \) are the collections of all open, closed, and clopen subsets of \( X \).
We write \( \Bor ( X ) \) for the collection of all Borel subsets of \( X \).

If \( A \subseteq X \times \pre{\omega}{\omega} \), then 
\[
\proj A = \setofLR{x \in X}{ \EXISTS{y \in \pre{\omega}{\omega} } ( x , y ) \in A } 
\] 
is the \markdef{projection} of \( A \).
The projective hierarchy on \( X \) is defined as follows: \( \bSigma\sp {1}\sb {1} ( X ) = \setof{ \proj C}{ C \in \bPi\sp {0}\sb {1} ( X \times \pre{\omega}{\omega} ) } \), 
 \( \bPi\sp {1}\sb {n} ( X ) = \setofLR{ X \setminus A}{ A \in \bSigma\sp {1}\sb {n} ( X ) } \), and \( \bSigma\sp {1}\sb {n + 1} ( X ) = \setofLR{ \proj C }{ C \in \bPi\sp {1}\sb {n} ( X \times \pre{\omega}{\omega} ) } \).
We also set \( \bDelta\sp {1}\sb {n} ( X ) = \bPi\sp {1}\sb {n} ( X ) \cap \bSigma\sp {1}\sb {n} ( X ) \).
The sets in \( \bSigma\sp {1}\sb {1} ( X ) \) are called \markdef{analytic} subsets of \( X \), and their complements, i.e.~the elements of \( \bPi\sp {1}\sb {1} ( X ) \), are called \markdef{coanalytic} subsets of \( X \), and by Lusin's theorem \( \bDelta\sp {1}\sb {1} ( X ) = \Bor ( X ) \).

A \markdef{boldface pointclass} for \( X \) is a family \( \bGamma ( X ) \subseteq \Pow ( X ) \) closed under continuous pre-images; its \markdef{dual} is the boldface pointclass \( \dual{ \bGamma } ( X ) = \setof{ X \setminus A }{ A \in \bGamma ( X ) } \).
If the space \( X \) is clear from the context we write 
\[ 
A\sp \complement = X \setminus A 
\] 
for the \markdef{complement of \( A \) in \( X \)}. 
We say that \( \bGamma \) is \markdef{self-dual} if it coincides with its dual, otherwise it is said to be non-self-dual.
The families \( \bSigma\sp {0}\sb { \alpha } ( X ) \), \( \bPi\sp {0}\sb { \alpha } ( X ) \), \( \bSigma\sp {1}\sb {n} ( X ) \), and \( \bPi\sp {1}\sb {n} ( X ) \) are non-self-dual boldface pointclasses when \( X \) is uncountable, while \( \bDelta\sp {0}\sb { \alpha } ( X ) \) and \( \bDelta\sp {1}\sb {n} ( X ) \) are self-dual.
The pointclass \( 2 \)-\( \bGamma ( X ) \) is \( \setof{ A \cap B }{A \in \bGamma ( X ) \wedge B \in \dual{ \bGamma} ( X ) } \), while its dual \( ( 2 \text{-} \bGamma ( X ) ) \dual{}\) is \( \setof{ A \cup B }{A \in \bGamma ( X ) \wedge B \in \dual{ \bGamma} ( X ) } \).
In particular, \( 2 \text{-} \bGamma ( X ) = 2 \text{-} \dual{\bGamma} ( X ) \).
(This is the first level of the Hausdorff's difference hierarchy~\cite[][Section 22.E]{Kechris:1995kc}.)
For the sake of brevity we will say that a subset \( A \) of \( X \) is in \( \bGamma \) to mean that \( A \in \bGamma ( X ) \).

If \( A \subseteq X \times Y \) and \( ( \bar{x} , \bar{y} ) \in X \times Y \) then the \markdef{vertical section} of \( A \) at \( \bar{x} \) and the \markdef{horizontal section} of \( A \) at \( \bar{y} \) are the sets
\[
 \vsection{A}{\bar{x}} = \setofLR{ y \in Y }{ ( \bar{x} , y ) \in A } \quad\text{and}\quad \hsection{A}{\bar{y}} = \setofLR{ x \in X }{ ( x , \bar{y} ) \in A } .
\]

\subsubsection{Effective methods}
All the Polish spaces considered in this paper (\( \omega \), the unit interval, \( \pre{\omega}{2} \), \( \KK \), \( \MALG \), \dots) are recursively presented, and so are their products.
Following~\cite[][Chapter 3]{Moschovakis:2009fk}, if \( X \) is a recursively presented Polish space, we can define the \markdef{lightface pointclasses} 
\[ 
\varSigma\sp 0\sb 1 \cap \Pow ( X ) , \quad \varPi\sp 0\sb 1 \cap \Pow ( X ) , \quad \varDelta\sp 0\sb 1 \cap \Pow ( X ) 
\] 
of the effectively open, closed, and clopen subsets of \( X \).
The effective analogue of the families of analytic, coanalytic, and Borel sets are 
\begin{align*}
\varSigma\sp 1\sb 1 \cap \Pow ( X ) & = \setof{ \proj C }{ C \in \varPi\sp 0\sb 1 \cap \Pow ( X \times \pre{\omega}{\omega} ) }
\\
 \varPi\sp 1\sb 1 \cap \Pow ( X ) & = \setof{ X \setminus A }{ A \in \varSigma\sp {1}\sb {1} \cap \Pow ( X ) } 
 \\ 
 \varDelta\sp 1\sb 1 \cap \Pow ( X ) & = \varSigma\sp 1\sb 1 \cap \varPi\sp 1\sb 1 \cap \Pow ( X ) .
\end{align*}
These lightface pointclasses can be relativized to any parameter \( p \in \pre{\omega}{\omega} \): if \( \varGamma \) is either one of \( \varSigma\sp i\sb n \), \( \varPi\sp i\sb n \) with \( i = 0 , 1 \) and \( n \geq 1 \), then set
\[ 
\varGamma ( p ) \cap \Pow ( X ) = \setof{ \vsection{A}{p} }{ A \in \varGamma \cap \Pow ( \pre{\omega}{\omega} \times X )} .
\] 
Therefore \( \bGamma ( X ) = \bigcup\sb {p \in \pre{\omega}{\omega} } \varGamma ( p ) \cap \Pow ( X ) \), where \( ( \bGamma , \varGamma ) \) is either one of the pairs \( ( \bSigma\sp {i}\sb {n} , \varSigma\sp i\sb n ) \) or \( ( \bPi\sp {i}\sb {n} , \varPi\sp i\sb n ) \) with \( i \leq 1 \leq n \).
It follows that \( \Bor ( X ) = \bDelta\sp {1}\sb {1} ( X ) = \bigcup\sb {p \in \pre{\omega}{\omega} } \varDelta\sp 1\sb 1 ( p ) \cap \Pow ( X ) \), where \( \varDelta\sp 1\sb 1 ( p ) = \varSigma\sp 1\sb 1 ( p ) \cap \varPi\sp 1\sb 1 ( p ) \).

Note that the lightface pointclasses can be used to classify points of the recursively presented Polish spaces.
In particular we can consider \( \varGamma ( p ) \cap \pre{\omega}{\omega} \) the collections of points of \( \pre{\omega}{\omega} \) which are in the relativization of \( \varGamma \) to \( p \in \pre{\omega}{\omega} \). 

\subsubsection{Complete sets}
A set \( A \subseteq X \) is \markdef{\( \bGamma \)-hard} if for each zero-dimensional Polish space \( Z \) and each \( B \in \bGamma ( Z ) \) there is a continuous \( f \colon Z \to X \) such that \( B = f\sp {-1} ( A ) \)---such a function is called a \markdef{reduction of \( A \) to \( B \)}.
If moreover \( A \in \bGamma ( X ) \), then \( A \) is said to be \markdef{\( \bGamma \)-complete}.
A set \( \mathcal{U} \in \bGamma ( X \times \pre{\omega}{ I } ) \) is \markdef{\( \bGamma \)-universal} (with respect to \( \pre{\omega}{ I } \)) if \( \bGamma ( X ) = \setof{ \hsection{\mathcal{U}}{y}}{ y \in \pre{\omega}{I} } \).

The boldface pointclasses \( \bSigma\sp {0}\sb { \alpha } ( X ) \), \( \bPi\sp {0}\sb { \alpha } ( X ) \), \( \bSigma\sp {1}\sb {n} ( X ) \), \( \bPi\sp {1}\sb {n} ( X ) \) have universal sets with respect to \( \pre{\omega}{ 2 } \) and to \( \pre{\omega}{\omega} \), and hence have complete sets.
Moreover, if \( X \) is recursively presented and if \( i \in \set{ 0 , 1 } \) and \( n \geq 1 \), then there is \( \mathcal{U} \in \varSigma\sp i\sb n ( X \times \pre{\omega}{I} ) \) which is universal for \( \bSigma\sp {i}\sb {n} ( X ) \), and similarly for \( \varPi\sp i\sb n \) and \( \bPi\sp {i}\sb {n} \).

If in the definition of \( \bGamma \)-completeness the function \( f \) witnessing \( \bGamma \)-hardness is only assumed to be Borel we have the weaker notion of \markdef{Borel-\( \bGamma \)-completeness}.
Assuming projective determinacy, for every \( n \geq 1 \) the notions of Borel-\( \bPi\sp {1}\sb {n} \)-completeness and Borel-\( \bSigma\sp {1}\sb {n} \)-completeness are equivalent to ordinary \( \bPi\sp {1}\sb {n} \)-completeness and \( \bSigma\sp {1}\sb {n} \)-completeness, and by a theorem of Kechris~\cite{Kechris:1997qy} the result holds in \( \ZFC \) when \( n = 1 \).

The set
\[ 
\IF = \setof{ T \in \Tr }{ \body{ T } \neq \emptyset } 
\] 
of all trees on \( \omega \) that have a branch is \( \varSigma\sp 1\sb 1 \) and \( \bSigma\sp { 1}\sb { 1} \)-complete, so its complement 
\[ 
\WF = \Tr \setminus \IF 
\] 
is \( \varPi\sp 1\sb 1 \) and \( \bPi\sp { 1 }\sb { 1 } \)-complete.
If \( U \) is \( \bGamma \)-complete and \( V \) is \( \dual{\bGamma} \)-complete, then \( U \times V \) is \( 2 \)-\( \bGamma \)-complete.
In particular: \( \WF \times \IF \) is \( 2 \)-\( \varSigma\sp {1}\sb {1} \) and \( 2 \)-\( \bSigma\sp {1}\sb {1} \)-complete.

Great many proofs of the fact that a given subset \( A \) of a Polish space \( X \) is \( \bSigma\sp {1}\sb {1} \)-hard rely on the construction of a continuous function \( \Tr \to X \) mapping ill-founded trees into \( A \) and well-founded trees outside of \( A \).
However, it is enough to continuously map trees with uncountably many branches inside \( A \) and well-founded trees outside of \( A \).
This can be seen by composing a reduction with the continuous, injective function \( \boldsymbol{E} \colon \Tr \to \Tr \) defined by
\[
\boldsymbol{E} ( T ) = {\downarrow} \setof{ \seq{ t\sb 0 , k\sb 0 , \dots , t\sb n , k\sb n } }{ \seq{ t\sb 0 , \dots , t\sb n } \in T \wedge k\sb 0 , \dots , k\sb n \in \omega }.
\]
The map \( \boldsymbol{E} \) enlarges the number of branches, in the sense that if \( T \) is well-founded then so is \( \boldsymbol{E} ( T ) \), and if \( T \in \IF \) then \( \body{ \boldsymbol{E} ( T ) } \) contains a perfect set.

If \( \kappa \leq \omega \) and \( \bowtie \) is one of \( {=} , {<} , { \leq } , {>} , {\geq} \), the set of all trees on \( \omega \) that have \( \bowtie \kappa \) branches is
\[ 
\Br{ \bowtie \kappa }{} = \setof{ T \in \Tr }{ \card{ \body{ T } } \bowtie \kappa } .
\]
For notational ease we will write \( \Br{ \kappa }{} \) rather than \( \Br{ = \kappa }{} \).
Note that \( \WF = \Br{ 0 }{} \) and that \( \IF = \Br{ \geq 1}{} \).
For every \( 1 \leq \kappa \leq \omega \) the set \( \Br{ \kappa }{} \) is the set of sections of the Borel set \( \setof{ ( T , x ) }{ x \in \body{T}} \) with exactly \( \kappa \) elements, so it is \( \bPi\sp {1}\sb {1} \) by Theorem~\ref{th:sections} below.
In fact it is \( \bPi\sp {1}\sb {1} \)-complete: if \( S \in \Br{\kappa}{} \) then the map \( \Tr \to \Tr \), \( T \mapsto 0 \conc S \cup 1 \conc \boldsymbol{E} ( T ) \) witnesses that \( \WF \leqW \Br{ \kappa }{} \).
Similarly \( \Br{ \leq n}{} \), \( \Br{ < \omega }{} \), and \( \Br{ \leq \omega }{} \) are \( \bPi\sp {1}\sb {1} \)-complete.

By König's lemma every infinite \( T \in \Tr\sb 2 \) has a branch, so the definitions above need to be changed.
Then let 
\begin{align*}
\IF \sb 2 &= \setof{ T \in \PrTr \sb 2 }{ \body{ T } \cap \mathcal{N} \neq \emptyset } ,
\\
\WF\sb 2 & = \setof{ T \in \PrTr\sb 2}{ \body{ T } \cap \mathcal{N} = \emptyset } ,
\\
\Br{ \bowtie \kappa }{2} & = \setof{ T \in \PrTr\sb 2}{ \card{ \body{ T } \cap \mathcal{N} } \bowtie \kappa } .
\end{align*}
The map \( \boldsymbol{E} \) above can be turned into a continuous, injective function 
\begin{equation}\label{eq:Explode}
\boldsymbol{E}\sb 2 \colon \PrTr\sb 2 \to \PrTr\sb 2 
\end{equation}
such that if \( T \in \WF\sb 2 \) then \( \boldsymbol{E}\sb 2 ( T ) \in \WF\sb 2 \), and if \( T \in \IF\sb 2 \) then \( \body{ \boldsymbol{E}\sb 2 ( T ) } \cap \mathcal{N} \) contains a perfect set. 
Thus by the arguments above 
\begin{equation}\label{eq:Brarecomplete}
 \Br{ n }{2} ,\Br{ \leq n }{2} ,\Br{ < \omega}{2} , \Br{ \leq \omega}{2} \text{ are \( \bPi\sp {1}\sb {1} \)-complete,}
\end{equation}
and \( \WF\sb 2 \times \IF\sb 2 \) is complete for \( 2 \)-\( \bSigma\sp {1}\sb {1} \).

The pruned tree \( \boldsymbol{E}\sb 2 ( T ) \) is obtained from \( T \) by pairing each \( x \in \body{ T } \cap \mathcal{N} \) with every \( y \in \pre{\omega}{\omega} \).
To be more specific: for \( t \in \pre{ < \omega }{2} \) and \( u \in \pre{ < \omega }{ \omega } \) with \( \LENGTH ( t ) = \lh ( u ) \), where \( \LENGTH \) is as in~\eqref{eq:virtuallength}, the sequence \( t \ltimes u \in \pre{ < \omega }{2} \) is defined as follows: if \( t = 0\sp { ( n\sb 0 ) } \conc 1 \conc 0\sp { ( n\sb 1 ) } \conc 1 \conc \dots \conc 1 \conc 0\sp { ( n\sb k ) } \conc 1 \conc 0\sp { ( n\sb {k + 1} ) } \) and \( u = \seq{ m\sb 0 , m\sb 1 , \dots , m\sb k } \)
\[
t \ltimes u = 0\sp { ( n\sb 0 ) } \conc 1 \conc 0\sp { ( m\sb 0 ) } \conc 1 \conc 0\sp { ( n\sb 1 ) } \conc 1 \conc 0\sp { ( m\sb 1 ) } \conc 1 \conc \dots \conc1 \conc 0\sp { ( n\sb k ) } \conc 1 \conc 0\sp { ( m\sb k ) } \conc 1 \conc 0\sp { ( n\sb {k + 1} ) } .
\]
Then let
\[
\boldsymbol{E}\sb 2 ( T ) = {\downarrow} \setof{ t \ltimes u }{ t \in T \wedge u \in \pre{ < \omega }{ \omega } \wedge \LENGTH ( t ) = \lh ( u ) } .
\]

\subsubsection{Some results on sections}\label{subsubsec:Sections}
The next result summarizes two classical results in Descriptive Set Theory and some easy consequences.

\begin{theorem}\label{th:sections}
Suppose \( X , Y \) are Polish, \( A \subseteq X \times Y \) is analytic, and for \( \kappa \leq \omega \) let 
\[ 
P\sb { \bowtie \kappa } = \setof{ x \in X }{ \card{\vsection{A}{x} } \bowtie \kappa } ,
\]
where \( \bowtie \) is one of \( {=} , {\leq} , {<} \).
\begin{enumerate-(a)}
\item\label{th:sections-a}
The set \( P\sb { \leq \omega } \) is coanalytic (Mazurkievicz-Sierpiński), and so are \( P\sb { \leq n } \) and \( P\sb { < \omega} \).
\item\label{th:sections-b}
If moreover \( A \) is Borel, then the set of uniqueness \( P\sb {=1} \) is also coanalytic (Lusin), and so are \( P\sb {= n} \) and \( P\sb {= \omega} \).
\end{enumerate-(a)}
\end{theorem}

\begin{proof}
The proof of the Mazurkievicz-Sierpiński and Lusin results can be found in~\cite[][Theorems 29.19 and 8.11]{Kechris:1995kc}.
Let \( A \) be analytic and let \( C \subseteq X \times Y \times \pre{\omega}{\omega} \) be a closed set that projects onto \( A \).
Then 
\[
x \in P\sb { \leq n } \IFF \neg \EXISTS{ ( y\sb 0 , z\sb 0 ), \dots , ( y\sb n , z\sb n )}\Bigl ( \bigwedge\sb { i \leq n } ( x , y\sb i , z\sb i ) \in C \wedge \bigwedge\sb { i < j \leq n } y\sb i \neq y\sb j \Bigr ) ,
\]
so \( P\sb { \leq n } \) is \( \bPi\sp {1}\sb {1} \).
Therefore \( P\sb {< \omega } = \bigcup\sb {n} P\sb { \leq n } \) is also \( \bPi\sp {1}\sb {1} \).

Suppose now \( A \) is Borel, and let \( < \) be a Borel linear order on \( Y \)---such order exists as \( Y \) is Borel isomorphic to either a countable discrete space, or else to \( \R \).
For \( n > 1 \) then \( P\sb {=n} \) is the set of uniqueness of the Borel set
\[
\setof{ ( x , ( y\sb 1 , \dots , y\sb n ) ) \in X \times Y\sp n }{ \set{ y\sb 1 , \dots , y\sb n } \subseteq \vsection{A}{ x } \AND y\sb 1 < \dots < y\sb n } . 
\]
We are left to prove that \( P\sb {=\omega} \in \bPi\sp {1}\sb {1} \).
The result is trivial if \( X \) is countable, and if \( Y \) were countable, then by~\cite[Lemma 18.12 and Exercise 18.15]{Kechris:1995kc} \( A \) would be the union of countably many graphs of Borel functions \( f\sb n \) uniformizing \( A \), so \( \proj A \) and each \( P\sb {=n} \) would be Borel, and so would be \( P\sb {=\omega} \).
Thus without loss of generality we may assume that \( X = Y = \pre{\omega}{\omega} \).
As \( A \) is Borel, it is \( \varDelta\sp 1\sb 1 ( p ) \) for some \( p \in \pre{\omega}{\omega} \).
By Harrison's perfect set theorem~\cite[Theorem 4F.1]{Moschovakis:2009fk}, any countable \( \varSigma\sp 1\sb 1 ( z ) \) subset of \( X \) contains only \( \varDelta\sp 1\sb 1 ( z ) \)-points, so for all \( x \in X = \pre{\omega}{\omega} \)
\begin{multline*}
 x \in P\sb {=\omega} \IFF x \in P\sb { \leq \omega } \wedge \FORALL{n} \EXISTS{ y \in \varDelta\sp 1\sb 1 ( x \oplus p ) \cap \pre{\omega}{\omega} } 
 \\
 [ \FORALL{i < j < n} ( \decode{y}{i} \neq \decode{y}{j} ) \wedge \FORALL{i < n} ( ( x , \decode{y}{i} ) \in A ) ] ,
\end{multline*}
where \( x \oplus p \) and \( \decode{y}{i} \) are as in~\eqref{eq:oplus} and~\eqref{eq:decode}.
By Kleene's theorem on restricted quantification~\cite[Theorem 4D.3]{Moschovakis:2009fk} the quantification \( \EXISTS{y \in \varDelta\sp 1\sb 1 ( x \oplus p ) \cap \pre{\omega}{\omega} } \) in the formula above is equivalent to a universal quantification, so \( P\sb {= \omega} \in \bPi\sp {1}\sb {1} \).
\end{proof}

Note that we cannot expect that either \( P\sb {=n} \) or \( P\sb {=\omega} \) be coanalytic, or even Borel, if \( A \in \bSigma\sp {1}\sb {1} \setminus \bDelta\sp {1}\sb {1} \).
For \( P\sb {=\omega} \) consider the set \( A = S \times T \) where \( S \subseteq X \) is analytic but not Borel and \( T \) is a countable subset of \( Y \); replacing \( T \) with a finite set yields a counterexample for \( P\sb {=n} \).

Part~\ref{th:sections-a} of Theorem~\ref{th:sections} says that the set of all large sections of an analytic set is analytic, where ``large'' means uncountable, that is containing a perfect set, by the perfect set property for \( \bSigma\sp {1}\sb {1} \).
The next result, whose proof can be found in~\cite[][Theorems 29.31 and 29.36]{Kechris:1995kc}, says that similar statement holds for other notions of largeness.

\begin{theorem}\label{th:sections2}
Suppose \( X , Y \) are Polish, \( A \subseteq X \times Y \) is analytic, \( U \subseteq Y \) open, and \( \nu \) a Borel probability measure on \( Y \).
\begin{enumerate-(a)}
\item\label{th:sections2-a}
The sets \( \setof{ x \in X }{ \vsection{A}{x} \text{ is not meager in } U } \), \( \setof{ x \in X }{ \vsection{A}{x} \text{ is comeager in } U } \) are analytic (Novikov).
\item\label{th:sections2-b}
For every \( a \in [ 0 ; 1 ] \) the sets \( \setof{ x \in X }{ \nu ( \vsection{A}{x} ) > a } \) and \( \setof{ x \in X }{ \nu ( \vsection{A}{x} ) \geq a } \) are analytic (Kondô-Tugué).
\end{enumerate-(a)}
\end{theorem} 

We will use later the following known result (see for example~\cite[Theorem 1.2]{Camerlo:2002sf}). 

\begin{theorem}\label{th:sectionKechris}
Let \( X \) be a standard Borel space, let \( Y \) be a Polish space, and let \( A\subseteq X \times Y \) be Borel.
Fix also \( B \), a Borel subset of \( \KK ( Y ) \).
Then \( \setofLR{ x\in X }{ \vsection{A}{ x } \in B } \) is a coanalytic subset of \( X \).
\end{theorem}

\section{Some basic results}
The measure \( \mu \) is the unique probability measure on \( \pre{ \omega }{2} \) satisfying \( \mu ( \Nbhd\sb s ) = 2\sp {-\lh s } \) for all \( s \in \pre{ < \omega }{2} \).
Let \( \MEAS \) be the collection of all \( \mu \)-measurable sets, and let \( \NULL = \setofLR{ A \in \MEAS}{ \mu ( A ) = 0 } \).

For \( A , B \in \MEAS \) write \( A =\sb \mu B \) just in case \( A \subseteq\sb \mu B \) and \( B \subseteq\sb \mu A \), where \( A \subseteq\sb \mu B \IFF A \setminus B \in \NULL \).
Clearly \( =\sb \mu \) is an equivalence relation, \( \NULL = \setofLR{A\in \MEAS}{ A =\sb \mu \emptyset } \) is an ideal, and the quotient 
\[
\MALG = \frac{ \MEAS}{ \NULL} = \frac{ \Bor }{\Bor \cap \NULL} 
\]
is a complete boolean algebra, called the \markdef{measure algebra}.
It is a Polish space with distance \( \delta ( \eq{A} , \eq{B} ) = \mu ( A \symdif B ) \).
The measure \( \mu \) induces a function on the quotient \( \hat{ \mu } \colon \MALG \to [ 0 ; 1 ] \), \( \hat{ \mu } ( \eq{A} ) = \mu ( A ) \).

\subsection{The density function}
A set \( A \subseteq \pre{ \omega }{2} \) will always be assumed to be measurable.
For \( z \in \pre{ \omega }{2} \), the \markdef{density of \( z \) at \( A \)} is
\[
 \density\sb A ( z ) = \lim\sb { n \to \infty } \frac{ \mu ( A \cap \Nbhd\sb {z \restriction n} )}{ \mu ( \Nbhd\sb { z \restriction n } ) } . 
\]
Note that
\[
\frac{ \mu ( A \cap \Nbhd\sb {z \restriction n} )}{ \mu ( \Nbhd\sb { z \restriction n } ) } = 2\sp {n} \cdot \mu ( A \cap \Nbhd\sb {z \restriction n} ) = \mu ( \LOC{A}{ z \restriction n} ) .
\]
The \markdef{upper density} and \markdef{lower density} are 
\[
\upperdensity \sb A ( z ) = \limsup\sb { n \to \infty } \frac{ \mu ( A \cap \Nbhd\sb {z \restriction n} )}{ \mu ( \Nbhd\sb { z \restriction n } ) } ,\qquad \lowerdensity \sb A ( z ) = \liminf\sb { n \to \infty } \frac{ \mu ( A \cap \Nbhd\sb {z \restriction n} )}{ \mu ( \Nbhd\sb { z \restriction n } ) } ,
\]
and the \markdef{oscillation of \( z \) at} \( A \) is
\[
\oscillation\sb A ( z ) = \upperdensity \sb A ( z ) - \lowerdensity \sb A ( z ) .
\]
The limit \( \density\sb A ( z ) \) does not exist if and only if \( \oscillation\sb A ( z ) > 0 \), and in this case we say that \( z \) is \markdef{blurry} for \( A \).
If instead \( \density\sb A ( z ) \) exists but it is not \( 0 \) or \( 1 \), then we say that \( z \) is \markdef{sharp} for \( A \).
The sets of all points that are blurry or sharp for \( A \) are denoted by \( \Blur ( A ) \) and \( \Sharp ( A ) \), and \( \Exc ( A ) = \Blur ( A ) \cup \Sharp ( A ) \) is the set of all points that are \markdef{exceptional} for \( A \).
Note that if \( A =\sb \mu B \) then 
\begin{align*}
 \density\sb { A\sp \complement } ( z ) & = 1 - \density\sb A ( z ) ,
& \density\sb A ( z ) & = \density\sb B ( z ) ,
\\
 \oscillation\sb { A\sp \complement } ( z ) & = \oscillation\sb A ( z ) ,
&
\oscillation\sb A ( z ) & = \oscillation\sb B ( z ) ,
\end{align*}
in the sense that if one of the two sides of the equations exists, then so does the other one, and their values are equal.

The Lebesgue density theorem says that for all \( A \in \MEAS \) 
\[ 
\Phi ( A ) \equalsdef \setof{z \in A}{ \density\sb A ( z ) = 1 } \in \MEAS 
\]
and that \( A \symdif \Phi ( A ) \in \NULL \).
In fact \( \Phi ( A ) \in \bPi\sp 0\sb 3 \) for all \( A \).
The function \( \Phi \) is \( =\sb \mu \) invariant, so it induces a function \( \hat{ \Phi } \) on the measure algebra.

\subsection{Complexity of sets in \( \MALG \) and in \( \KK \)}\label{subsec:complexityofsetsinMALG}
From~\cite[]{Andretta:2015kq} we have that 
\[
f\sb n \colon \pre{\omega}{2} \times \MALG \to [ 0 ; 1 ] , \quad ( x , \eq{A} ) \mapsto \mu ( A\cap \Nbhd\sb {x \restriction n } ) 
\] 
is continuous for all \( n \in \omega \), while the three functions \( \upperdensity , \lowerdensity , \oscillation \colon \pre{\omega}{2} \times \MALG \to [ 0 ; 1 ] \) defined by
\begin{align*}
 & ( x , \eq{A} ) \mapsto \upperdensity\sb A ( x ) , && ( x , \eq{A} ) \mapsto \lowerdensity\sb A  ( x ) , && ( x , \eq{A} ) \mapsto \oscillation\sb A ( x ) 
\end{align*}
are in \( \Baireclass\sb 2 \).
If \( \MALG \) is replaced by \( \KK \), the resulting functions (which are still denoted by the same letter) \( \pre{\omega}{2}\times \KK \to [ 0 ; 1 ] \) are: \( f\sb n \in \Baireclass\sb 1 \) and \( \upperdensity , \lowerdensity , \oscillation \in \Baireclass\sb 3 \), since they are obtained by composing with the function 
\begin{equation}\label{eq:embeddingKintoMALG}
 j \colon \KK \to \MALG , \quad K \mapsto \eq{K} 
\end{equation}
which is in \( \Baireclass\sb 1 \), while \( \ran j \) is a \( \bPi\sp {0}\sb {3} \)-complete subset of \( \MALG \)---see~\cite[]{Andretta:2015kq} for details.

One can ask whether these results are sharp: seen as a function with domain \( \pre{\omega}{2} \times \KK \), the function \( f\sb n \) is not continuous, so it is indeed in \( \Baireclass\sb 1\setminus \Baireclass\sb 0 \), while in Corollary~\ref{cor:D^-} it is shown that \( \lowerdensity \in \Baireclass\sb 2 \setminus \Baireclass\sb 1 \).
Going back to \( \MALG \), the set
\[
\widehat{\mathcal{A}}\sb {\bowtie} = \setof{ ( \eq{A} , z , n , r ) \in \MALG \times \pre{\omega}{2} \times \omega \times [ 0 ; 1 ] }{ \mu ( A \cap \Nbhd\sb { z \restriction n} ) \bowtie r } 
\]
is \( \bSigma\sp 0\sb 1 \), where \( \bowtie \) stands for \( < \) or \( > \). 
If instead \( \bowtie \) denotes \( \leq \) or \( \geq \) then \( \widehat{\mathcal{A}}\sb {\bowtie} \in \bPi\sp {0}\sb {1} \) and hence \( \widehat{\mathcal{A}} = \setof{ ( \eq{A} , z , n , r ) \in \MALG \times \pre{\omega}{2} \times \omega \times [ 0 ; 1 ] }{ \mu ( A \cap \Nbhd\sb { z \restriction n} ) = r } \in \bPi\sp {0}\sb {1} \). 
Similarly \( \widehat{\mathcal{B}} = \setofLR{( \eq{A} , z , r ) \in \MALG \times \pre{\omega }{2} \times [ 0 ; 1 ] }{ \density \sb A ( z ) = r } \in \bPi\sp {0}\sb {3} \).

We now look at the analogous sets in \( \KK \).
Let
\begin{align*}
\mathcal{A}\sb {\bowtie} & = \setofLR{( K , z , n , r ) \in \KK \times \pre{\omega }{2} \times \omega \times [ 0 ; 1 ] }{\mu ( K \cap \Nbhd\sb {z \restriction n } ) \bowtie r } 
\\
\mathcal{B}\sp+ \sb { \bowtie } & = \setofLR{( K , z , r ) \in \KK \times \pre{\omega }{2} \times [ 0 ; 1 ] }{ \upperdensity \sb K ( z ) \bowtie r }
\\
\mathcal{B}\sp - \sb { \bowtie } & = \setofLR{( K , z , r ) \in \KK \times \pre{\omega }{2} \times [ 0 ; 1 ] }{ \lowerdensity \sb K ( z ) \bowtie r }
\end{align*}
where \( \bowtie \) is one of the ordering relations: \( < \), \( > \), \( \leq \), and \( \geq \), and let \( \mathcal{A} = \mathcal{A}\sb { \leq} \cap \mathcal{A}\sb { \geq } \) and \( \mathcal{B} = \mathcal{B}\sp + \sb { \leq} \cap \mathcal{B} \sp - \sb { \geq } \), that is 
\begin{align*}
\mathcal{A} & = \setofLR{( K , z , n , r ) \in \KK \times \pre{\omega }{2} \times \omega \times [ 0 ; 1 ] }{\mu ( K \cap \Nbhd\sb {z \restriction n } ) = r } 
\\
 \mathcal{B} & = \setofLR{( K , z , r ) \in \KK \times \pre{\omega }{2} \times [ 0 ; 1 ] }{ \density \sb K ( z ) = r } .
\end{align*}
Note that the complement of \( \mathcal{A}\sb {<} \) is \( \mathcal{A}\sb {\geq} \) and the complement of \( \mathcal{A}\sb {>} \) is \( \mathcal{A}\sb {\leq} \) (and similarly for \( \mathcal{B}\sp \pm\sb {\bowtie} \)), so we can cut-down the verifications in half when computing the complexity of these sets.

\begin{lemma}\label{lem:measureofcompactsinCantor}
\begin{enumerate-(a)}
\item\label{lem:measureofcompactsinCantor-a}
 \( \mathcal{A}\sb {<} \in \bSigma\sp {0}\sb {1} \) and hence \( \mathcal{A}\sb {\geq} \in \bPi\sp {0}\sb {1} \); \( \mathcal{A}\sb {\leq} \in \bPi\sp {0}\sb {2} \) and hence \( \mathcal{A}\sb {>} \in \bSigma\sp {0}\sb {2} \). 
Therefore \( \mathcal{A} \) is \( \bPi\sp {0}\sb {2} \).

\item\label{lem:measureofcompactsinCantor-b}
\( \mathcal{B}\sp -\sb {\leq} \) is \( \bPi\sp {0}\sb {2} \), \( \mathcal{B}\sp -\sb {\geq} , \mathcal{B}\sp +\sb {\geq } \) are \( \bPi\sp {0}\sb {3} \), and \( \mathcal{B}\sp +\sb {\leq} \) is \( \bPi\sp {0}\sb {4} \); therefore \( \mathcal{B}\sp -\sb {>} \) is \( \bSigma\sp {0}\sb {2} \), \( \mathcal{B}\sp -\sb {<} , \mathcal{B}\sp +\sb {<} \) are \( \bSigma\sp {0}\sb {3} \), and \( \mathcal{B}\sp +\sb { > } \) is \( \bSigma\sp {0}\sb {4} \).
Therefore \( \mathcal{B} \) is \( \bPi\sp {0}\sb {4} \).
\end{enumerate-(a)}
\end{lemma}

\begin{proof}
\ref{lem:measureofcompactsinCantor-a}
Let us check that \( \mathcal{A}\sb {<} \in \bSigma\sp 0\sb 1 \) and therefore \( \mathcal{A}\sb {\geq} \in \bPi\sp 0\sb 1 \).
Fix \( ( K , z , n , r ) \in \mathcal{A}\sb {<} \), that is to say: \( \mu ( K \cap \Nbhd\sb { z \restriction n } ) < r \).
We must find an open subset of \( \KK \times \pre{\omega}{2} \times \omega \times [ 0 ; 1 ] \) containing \( ( K , z , n , r ) \) and included in \( \mathcal{A}\sb {<} \).
As \( \omega \) is discrete and \( z \mapsto \mu ( K \cap \Nbhd\sb { z \restriction n } ) \) is locally constant, it is enough to show that \( ( K' , z , n , r' ) \in \mathcal{A}\sb {<} \) for all \( K' \) sufficiently close to \( K \) and all \( r' \) sufficiently close to \( r \).
Let \( U \) be open and such that \( K \cap \Nbhd\sb { z \restriction n} \subseteq U \subseteq \Nbhd\sb { z \restriction n} \) and \( \mu ( U ) < r \).
Then for all \( K' \) sufficiently close to \( K \) and all \( r' > \mu ( U ) \) it follows that \( \mu ( K' \cap \Nbhd\sb { z \restriction n} ) < r' \).

As 
\[ 
( K , z , n , r ) \in \mathcal{A}\sb { \leq } \IFF r = 1 \vee \FORALL{ \varepsilon } [ 0 < \varepsilon \leq 1 - r \implies ( K , z , n , r + \varepsilon ) \in \mathcal{A}\sb {<} ] ,
\] 
then \( \mathcal{A}\sb {\leq} \in \bPi\sp 0\sb 2 \) and \( \mathcal{A}\sb {>} \in \bSigma\sp 0\sb 2 \).

\smallskip

\ref{lem:measureofcompactsinCantor-b}
The result follows from part~\ref{lem:measureofcompactsinCantor-a} since
\begin{align*}
\lowerdensity\sb K ( z ) \geq r & \IFF \FORALL{ \varepsilon > 0 } \FORALLS{\infty}{n} [ \mu ( K \cap \Nbhd\sb { z \restriction n } ) / 2\sp {-n} \geq r - \varepsilon ] && \text{and so } \mathcal{B}\sp -\sb {\geq} \in \bPi\sp {0}\sb {3}
\\
\lowerdensity\sb K ( z ) > r & \IFF \EXISTS{ \varepsilon > 0 } \FORALLS{\infty}{n} [ \mu ( K \cap \Nbhd\sb { z \restriction n } ) / 2\sp {-n} \geq r + \varepsilon ] && \text{and so } \mathcal{B}\sp -\sb {>} \in \bSigma\sp {0}\sb {2}
\\
\upperdensity\sb K ( z ) < r & \IFF \EXISTS{ \varepsilon > 0 } \FORALLS{\infty}{n} [ \mu ( K \cap \Nbhd\sb {z \restriction n } ) / 2\sp {- n } \leq r - \varepsilon ] && \text{and so } \mathcal{B}\sp +\sb {<} \in \bSigma\sp {0}\sb {3} 
\\
\upperdensity\sb K ( z ) \leq r & \IFF \FORALL{ \varepsilon > 0} \FORALLS{\infty}{n} [ \mu ( K \cap \Nbhd\sb {z \restriction n } ) / 2\sp {- n } < r + \varepsilon ] && \text{and so } \mathcal{B}\sp +\sb { \leq } \in \bPi\sp {0}\sb {4} .
\end{align*}
(Note that in the last equivalence we could replace \( < \) with \( \leq \), but that would not help in reducing the complexity as \( \mathcal{A}\sb {\leq} \) is \( \bPi\sp {0}\sb {2} \).)
\end{proof}

\begin{theorem}\label{th:measureofcompactsinCantor}
\begin{enumerate-(a)}
\item\label{th:measureofcompactsinCantor-a}
For every \( r \in \cointerval{ 0 }{ 1 } \) and every \( z\in \pre{\omega}{2} \), the section 
\[ 
\mathcal{A}\sb {\leq} \sp * = \setofLR{ K \in \KK }{ ( K , z , 0 , r ) \in \mathcal{A}\sb { \leq} } = \setof{ K \in \KK }{ \mu ( K ) \leq r }
\]
is \( \bPi\sp {0}\sb {2} \)-complete, and therefore its complement \( \mathcal{A}\sb {>} \sp * \) is \( \bSigma\sp {0}\sb {2} \)-complete.
Therefore \( \mathcal{A}\sb {\leq} \) and \( \mathcal{A}\sb {>} \) are, respectively \( \bPi\sp {0}\sb {2} \)-complete, and \( \bSigma\sp {0}\sb {2} \)-complete. 
Moreover \( \mathcal{A} \) is \( \bPi\sp {0}\sb {2} \)-complete.

\item\label{th:measureofcompactsinCantor-b}
 \( \mathcal{B}\sp -\sb {\leq} \) is \( \bPi\sp {0}\sb {2} \)-complete, \( \mathcal{B}\sp -\sb {\geq} \) is \( \bPi\sp {0}\sb {3} \)-complete; therefore \( \mathcal{B}\sp -\sb {>} \) is \( \bSigma\sp {0}\sb {2} \)-complete, \( \mathcal{B}\sp -\sb {<} \) is \( \bSigma\sp {0}\sb {3} \)-complete.
Moreover \( \mathcal{B} \) is \( \bPi\sp {0}\sb {3} \)-hard.
\end{enumerate-(a)}
\end{theorem}

\begin{proof}
\ref{th:measureofcompactsinCantor-a}
Fix \( r \in \cointerval{ 0 }{ 1 } \).
We first prove that \( \mathcal{A}\sb {\leq}\sp * \) is \( \bPi\sp {0}\sb {2} \)-hard.
Let \( K\sb 0 \in \KK \) with \( \mu ( K\sb 0 ) = r \), and let \( s \in \pre{ < \omega }{ 2 } \) be such that \( \Nbhd\sb s \cap K\sb 0 = \emptyset \).
The function 
\[ 
h \colon \KK \to \KK , \quad K \mapsto K\sb 0 \cup s \conc K 
\] 
reduces \( \setof{ K \in \KK }{ \mu ( K ) = 0 } \) to \( \setof{ K \in \KK }{ \mu ( K ) = r } \subseteq \mathcal{A}\sb {\leq}\sp * \) and \( \setof{ K \in \KK }{ \mu ( K ) > 0 } \) to \( \setof{ K \in \KK }{ \mu ( K ) > r } \).
By Lemma~\ref{lem:measureofcompactsinCantor}\ref{lem:measureofcompactsinCantor-a}, \( \setof{ K \in \KK }{ \mu ( K ) = 0 } \) is \( \Gdelta \), and since it is dense and co-dense, it is actually \( \bPi\sp {0}\sb {2} \)-complete.
Therefore \( \mathcal{A}\sb {\leq}\sp * \) is \( \bPi\sp {0}\sb {2} \)-hard.
The same argument shows that \( \mathcal{A} \) is \( \bPi\sp {0}\sb {2} \)-complete.

\smallskip

\ref{th:measureofcompactsinCantor-b}
Fix \( r \in \cointerval{ 0 }{ 1 } \).
We show that \( \setof{ K \in \KK }{ \lowerdensity \sb K ( 0\sp {( \omega )} ) \leq r } \) is \( \bPi\sp {0}\sb {2} \)-complete, and therefore so is \( \mathcal{B}\sp -\sb { \leq } \).
We argued above that \( \setof{ K \in \KK }{ \mu ( K ) = 0 } \) is \( \bPi\sp {0}\sb {2} \)-complete, so it is enough to construct a continuous \( f \colon \KK \to \KK \) such that \( \mu ( K ) = 0 \iff \lowerdensity\sb { f ( K ) } ( 0\sp { ( \omega ) } ) \leq r \).
Using the function \( h \) from part~\ref{th:measureofcompactsinCantor-a} of the proof, let \( f ( K ) = \set{ 0\sp {( \omega ) } } \cup \bigcup\sb {n \in \omega } 0\sp {( n ) } \conc 1 \conc h ( K ) \).

For \( \mathcal{B}\sp -\sb { \geq } \) and \( \mathcal{B} \), notice that their section when \( r = 1 \) and \( K \) is of positive measure and empty interior, that is \( \setofLR{ z}{ \lowerdensity\sb K ( z ) = 1} = \Phi ( K ) \), is \( \bPi\sp {0}\sb {3} \)-complete by~\cite[][Theorem 1.3]{Andretta:2013uq}.
\end{proof}

\begin{question}
The sets \( \mathcal{B}\sp +\sb {\leq} \) and \( \mathcal{B} \) are \( \bPi\sp {0}\sb {4} \).
Are they \( \bPi\sp {0}\sb {4} \)-complete?
\end{question}

\begin{remarks}
\begin{enumerate-(a)}
\item
The set \( \setof{ K \in \KK }{ \mu ( K ) = 0 } \) is comeager in \( \KK \), and it is contained in \( \mathcal{C}\sb 1 = \setof{ K \in \KK }{ \Int K = \emptyset } \) and disjoint from \( \mathcal{C}\sb 2 = \setof{ K \in \KK }{ \Phi ( K ) \text{ is \( \bPi\sp {0}\sb {3} \)-complete} } \).
Thus \( \mathcal{C}\sb 1 \) is comeager, and \( \mathcal{C}\sb 2 \) is meager.
\item
The function \( \mu \colon \KK \to [ 0 ; 1 ] \) is upper semicontinuous, that is: if \( K\sb n \to K \) and \( \mu ( K\sb n ) \geq r \) then \( \mu ( K ) \geq r \).

To see this, associate to each \( T \in \PrTr\sb 2 \) and \( n \in \omega \) the number \( M ( T , n ) = \card{\Lev ( T , n ) } / 2\sp {n} \), where \( \Lev ( T , n ) \) is the set of nodes of \( T \) of length \( n \).
For each \( n \) the map \( T \mapsto M ( T , n ) \) is continuous, and \( \mu ( \body{T} ) = \inf\sb n M ( T , n ) \). 
\item
The complexity of the \( z \)-sections of \( \mathcal{B}\sp +\sb {\leq} \) does not depend on \( z \in \pre{\omega}{2} \), so in order to study their position in the Borel hierarchy, it is enough to focus on the section \( \mathcal{C} = \setof{ ( K , r )}{ ( K , 0\sp {( \omega )} , r ) \in \mathcal{B}\sp +\sb {\leq} } \).
(Apply an isometry of \( \pre{\omega}{2} \) sending \( z \) to \( 0\sp {( \omega )} \)).
Moreover the section \( \vsection{\mathcal{C}}{ K } \) is closed, being \( [ \upperdensity\sb K ( 0\sp {( \omega )} ) ; 1 ] \), while for every \( r \in \cointerval{0}{1} \) the section \( \hsection{\mathcal{C}}{ r } \) is \( \bPi\sp {0}\sb {2} \)-hard.

To see this use the map \( \KK \to \KK \), \( K \mapsto \set{ 0\sp {( \omega ) } } \cup \bigcup\sb {n \in \omega } 0\sp {( n )} \conc 1 \conc K \), which reduces the \( \bPi\sp {0}\sb {2} \)-complete set \( \setof{ K \in \KK }{ \mu ( K ) \leq r } \) to \( \setof{ K \in \KK }{ \density\sb K ( 0\sp {( \omega ) } ) \leq r } \).
\end{enumerate-(a)}
\end{remarks}

\begin{corollary}\label{cor:D^-}
The function \( \lowerdensity \colon \pre{\omega}{2} \times \KK \to [ 0 ; 1 ] \) is in \( \Baireclass\sb 2 \setminus \Baireclass\sb 1 \).
\end{corollary}

\begin{proof}
The preimage of \( ( a ; b ) \) via \( \lowerdensity \) is \( \setof{ ( z , K ) }{ ( K , z , b ) \in \mathcal{B}\sp {-}\sb {<} \wedge ( K , z , a ) \in \mathcal{B}\sp {-}\sb {>} } \) which is a set in \( \bSigma\sp {0}\sb {3} \), so the preimage of an open set is \( \bSigma\sp {0}\sb {3} \).
The preimage of \( \set{1} \) under \( \lowerdensity \) is \( \setof{ ( z , K ) }{ z \in \Phi ( K ) } \), which by the argument in the proof of Theorem~\ref{th:measureofcompactsinCantor}\ref{th:measureofcompactsinCantor-b} is \( \bPi\sp {0}\sb {3} \)-complete.
Therefore \( \lowerdensity \) is not in \( \Baireclass\sb 1 \).
\end{proof}

\begin{question}
Are the functions \( \upperdensity , \oscillation \colon \pre{\omega}{2} \times \KK \to [ 0 ; 1 ] \) in \( \Baireclass\sb 3 \setminus \Baireclass\sb 2 \)?
\end{question}

\subsection{Sets that are solid, dualistic, or spongy}
A set \( A \subseteq \pre{ \omega }{2} \) is said to be
\begin{itemize}
\item
\markdef{solid} iff \( \Blur ( A ) = \emptyset \),
\item
\markdef{quasi-dualistic} iff \( \Sharp (A ) = \emptyset \),
\item
\markdef{dualistic} iff it is quasi-dualistic and solid iff \( \Exc ( A ) = \emptyset \),
\item
\markdef{spongy} iff \( \Blur ( A ) \neq \emptyset = \Sharp (A ) \) iff it is quasi-dualistic but not solid.
\end{itemize}
The sets \( \Blur ( A ) \) and \( \Sharp ( A ) \) are \( \bSigma\sp {0}\sb {3} \) and \( \bPi\sp {0}\sb {3} \), respectively.
If \( A \) is dualistic, then \( \Phi ( A ) \) and \( \Phi ( A\sp \complement ) \) are \( \bDelta\sp {0}\sb {2} \) by~\cite[][Section 3.3]{Andretta:2013uq}. 
In~\cite[][Section 3.4]{Andretta:2013uq} examples of dualistic, solid, spongy sets are constructed.

The collections of sets that are solid, dualistic, quasi-dualistic, or spongy are denoted by \( \Solid \), \( \Dual \), \( \qDual \), and \( \Spongy \).
Also
\[
\bDelta\sp {0}\sb {1} \subseteq \Dual = \Solid \cap \qDual .
\]
One can further refine this taxonomy of measurable sets by imposing some restriction on the number of blurry/sharp points and on the number of values that the density function can attain.
For example, for \( \kappa \leq \omega \) and \( \bowtie \) one of \( {<} , {\leq} , {>} , {\geq} \) or \( = \)  one can consider 
\begin{align*}
\BLR \sb { \bowtie \kappa } & = \setof{ A \in \MEAS }{ \card{ \Blur ( A ) } \bowtie \kappa }
\\
\SHARP\sb { \bowtie \kappa } & = \setof{ A \in \MEAS }{ \card{ \Sharp ( A ) } \bowtie \kappa }
\\
\Range\sb { \bowtie \kappa } & = \setof{ A \in \MEAS }{ \card{ \ran \density\sb A \cap ( 0 ; 1 ) } \bowtie \kappa } .
\end{align*}
For the sake of readability the \( = \) sign will be dropped from the subscript and we write \( \BLR \sb { \kappa } , \SHARP\sb \kappa , \Range\sb \kappa \). 
Thus 
\begin{equation}\label{eq:dual}
\begin{aligned}
 \Solid &= \BLR\sb {0} , & \qDual & = \SHARP\sb {0} = \Range\sb {0} , 
 \\ 
 \Spongy & = \qDual \cap \BLR\sb {\geq 1} , & \Dual &= \BLR\sb {0} \cap \SHARP\sb {0} . 
\end{aligned}
\end{equation}
One can also consider the class of all measurable sets such that the density function is injective (on the sharp points), or attains values in a given set \( S \subseteq ( 0 ; 1 ) \), or attains either meager-many or null-many values:
\begin{align*}
\Range\sb {\mathrm{inj}} &= \setof{ A \in \MEAS }{ \FORALL{z\sb 1 , z\sb 2 \in \pre{\omega}{2} }( \density\sb A ( z\sb 1 ) = \density\sb A ( z\sb 2 ) \in ( 0 ; 1 ) \implies z\sb 1 = z\sb 2) }
 \\
\Range\sb {\MGR} & = \setof{ A \in \MEAS }{ \ran ( \density\sb { A } ) \in \MGR }
\\
\Range ( S ) & = \setof{ A \in \MEAS }{ \ran ( \density\sb A ) = S \cup \setLR{ 0 , 1 } } 
\\
\Range\sb { \lambda \bowtie a } & = \setof{ A \in \MEAS }{\lambda ( \ran ( \density\sb { A } ) ) \bowtie a } 
 \end{align*}
where \( a \in [ 0 ; 1] \), the symbol \( \bowtie \) denotes one of the relations \( \leq , < , \geq , > , = \), and \( \lambda \) is the Lebesgue measure on \( \R \).
It is easy to check that \( \ran ( \density\sb A ) \cap ( 0 ;1 ) \) is \( \bSigma\sp {1}\sb {1} \), so \( \Range ( S ) = \emptyset \) whenever \( S \) is not analytic.

All these families of sets are invariant under \( =\sb \mu \), so they can be defined on the measure algebra as well, that is to say: for \( \mathcal{C} \) one of the collections above, let \( \widehat{\mathcal{C}} = \setof{ \eq{A} \in \MALG }{ A \in \mathcal{C}} \).
By~\cite[]{Andretta:2015kq} \( \widehat{\Spongy} \) is comeager in \( \MALG \).

For \( S \subseteq ( 0 ; 1 ) \), let \( \widehat{\Range} ( { \subseteq } S ) = \setof{ \eq{ A } \in \MALG }{ \ran \density\sb A \subseteq S \cup \set{ 0 , 1} } \), and define \( \widehat{\Range} ( { \supseteq } S ) \) similarly.

\begin{lemma}\label{lem:rangecontainment}
Let \( S \subseteq ( 0 ; 1 ) \) be \( \bSigma\sp {1}\sb {1} \).
\begin{enumerate-(a)}
\item\label{lem:rangecontainment-a}
\( \widehat{\Range} ( { \subseteq } S ) \) and \( \widehat{\Range} ( { \supseteq } S ) \) are \( \bPi\sp {1}\sb {2} \).
\item\label{lem:rangecontainment-b}
If \( S \) is Borel then \( \widehat{\Range} ( { \subseteq } S ) \) is \( \bPi\sp {1}\sb {1} \).
\item\label{lem:rangecontainment-c}
If \( S \) is countable then \( \widehat{\Range} ( { \supseteq } S ) \) is \( \bSigma\sp {1}\sb {1} \).
\end{enumerate-(a)} 
\end{lemma}

\begin{proof}
Parts \ref{lem:rangecontainment-a} and \ref{lem:rangecontainment-b} follow by 
\begin{align*}
\eq{A} \in \widehat{\Range} ( { \subseteq } S ) & \IFF \FORALL{z \in \pre{\omega}{2} }( \oscillation\sb A ( z ) = 0 \implies \density\sb A ( z ) \in S \cup \set{ 0 , 1 } )
\\
\eq{A} \in \widehat{\Range} ( { \supseteq } S ) & \IFF \FORALL{r \in S } \EXISTS{ z \in \pre{\omega}{2} } ( \density\sb A ( z ) = r ) .
\end{align*}
For \ref{lem:rangecontainment-c} we may assume that \( S \neq \emptyset \) otherwise the result is trivial.
So let \( \seqofLR{ r\sb n }{ n \in \omega } \) be an enumeration (possibly with repetitions) of \( S \).
Then \( \eq{A} \in \widehat{\Range} ( { \supseteq } S ) \) iff \( \FORALL{n} \EXISTS{ z \in \pre{\omega}{2} } ( \density\sb A ( z ) = r\sb n ) \).
\end{proof}

\begin{theorem}\label{th:upperbound}
Let \( n < \omega \).
\begin{enumerate-(a)}
\item\label{th:upperbound-a}
The following collections of sets are \( \bPi\sp {1}\sb {1} \):
\begin{enumerate}[label={\upshape (a\arabic*)}, leftmargin=2pc]
\item\label{th:upperbound-a-1}
\( \widehat{\SHARP}\sb { n } \), \( \widehat{\SHARP}\sb { \leq n} \), \( \widehat{\SHARP}\sb { < \omega } \), \( \widehat{\SHARP}\sb { \leq \omega } \), \( \widehat{\SHARP}\sb { \omega } \),
\item\label{th:upperbound-a-2}
\( \widehat{\BLR}\sb {n } \), \( \widehat{\BLR}\sb {\leq n } \), \( \widehat{\BLR}\sb { < \omega} \), \( \widehat{\BLR}\sb {\leq \omega} \), \( \widehat{\BLR}\sb { \omega} \),
\item\label{th:upperbound-a-3}
 \( \widehat{\Range}\sb { \leq n } \), \( \widehat{\Range}\sb { < \omega } \), \( \widehat{\Range}\sb { \leq \omega } \), \( \widehat{\Range}\sb {\mathrm{inj}} \), \( \widehat{\Range}\sb {\MGR} \), \( \widehat{\Range}\sb { \lambda \leq a } \) and \( \widehat{\Range}\sb { \lambda < a } \) for any \( a \in [ 0 ; 1 ] \),
\item\label{th:upperbound-a-4}
\( \widehat{\Solid} \), \( \widehat{\qDual} \), and \( \widehat{\Dual} \).
\end{enumerate}
\item\label{th:upperbound-b}
The following collections of sets are \( 2 \)-\( \bSigma\sp {1}\sb {1} \): \( \widehat{\Range} \sb { n + 1 } \), \( \widehat{\Range} \sb { \omega } \), \( \widehat{\Range}\sb { \lambda = a } \), \( \widehat{\Range} ( S ) \) for \( a \in [ 0 ; 1 ] \) and \( S \subseteq ( 0 ; 1 ) \) countable, and \( \widehat{\Spongy} \).
\item\label{th:upperbound-c}
\( \widehat{\Range} ( S ) \) is \( \bPi\sp {1}\sb {2} \), for any \( S \subseteq ( 0 ; 1 ) \). 
\end{enumerate-(a)}
 
Therefore in the space \( \KK \) the homologous sets \( \SHARP\sb { n } \), \dots{} have the same complexity.
\end{theorem}

\begin{proof}
 \ref{th:upperbound-a}:
Let \( G \) be one of \( \lowerdensity , \upperdensity , \oscillation \).
 By~\cite{Andretta:2015kq} the function \( \MALG \times \pre{\omega}{2} \to [ 0 ; 1 ] \), \( ( \eq{A} , x ) \mapsto G\sb A ( x ) \) is Borel (in fact in \( \Baireclass\sb 2 \)).
Therefore
\[
 \mathcal{B} = \setofLR{ ( \eq{A} , x ) \in \MALG \times \pre{\omega}{2} }{ \density\sb A ( x ) \in ( 0 ; 1 ) }
\]
 is Borel, so Theorem~\ref{th:sections} yields that \( \widehat{\SHARP}\sb { \bowtie \kappa } = \setofLR{ \eq{A} }{ \card{ \vsection{\mathcal{B}}{ \eq{A}}} \bowtie \kappa } \) is \( \bPi\sp {1}\sb {1} \), where \( \kappa \leq \omega \) and \( \bowtie \) is one of \( = , \leq , < \).
 Thus \ref{th:upperbound-a-1} holds.
 
 For \ref{th:upperbound-a-2} apply the same argument with \( \mathcal{B} = \setof{ ( \eq{A} , x ) \in \MALG \times \pre{\omega}{2} }{ \oscillation\sb A ( x ) > 0 } \).

Now for~\ref{th:upperbound-a-3}: the set 
\[
 \mathcal{A} = \setofLR{ ( \eq{A} , r ) \in \MALG \times ( 0 ; 1 ) }{ \EXISTS{x \in \pre{\omega}{2} } ( \density\sb A ( x ) = r ) }
\]
is \( \bSigma\sp {1}\sb {1} \), so Theorem~\ref{th:sections} implies that \( \widehat{\Range}\sb { \leq n } \), \( \widehat{\Range}\sb { < \omega } \), \( \widehat{\Range}\sb { \leq \omega } \) are \( \bPi\sp {1}\sb {1} \).
By inspection \( \widehat{\Range}\sb {\mathrm{inj}} \) is \( \bPi\sp {1}\sb {1} \), while the result for \( \widehat{\Range}\sb { \lambda \leq a } \), \( \widehat{\Range}\sb { \lambda < a } \), and \( \widehat{\Range}\sb {\MGR} \) follows from Theorem~\ref{th:sections2}.

Part~\ref{th:upperbound-a-4} follows by~\eqref{eq:dual}.

\smallskip

For parts \ref{th:upperbound-b} and \ref{th:upperbound-c} argue as follows.
The complexity of \( \widehat{\Range} \sb { n + 1 } \), \( \widehat{\Range} \sb { \omega } \), and \( \widehat{\Range}\sb { \lambda = a } \) follows from~\ref{th:upperbound-a-3}, while the complexity of \( \widehat{\Spongy} \) is established by inspection.
The remaining cases follow from Lemma~\ref{lem:rangecontainment}.
\end{proof}

\begin{proposition}\label{prop:dualisticofallmeasures}
For every \( 0 < r < 1 \) there is a dualistic open set \( U \) such that \( \mu ( U ) = \mu ( \Cl U ) = r \), so that \( \Cl U \) is also dualistic.
\end{proposition}

We need a preliminary result.

\begin{lemma}\label{lem:dyadicvalueofdualistic}
Let \( \emptyset \neq U \subseteq \pre{ \omega }{2} \) be open, and let \( d \) be dyadic, with \( 0 < d < \mu ( U ) \).
Then there exists a clopen set \( V \subseteq U \) such that \( \mu ( V ) = d \).
\end{lemma}

\begin{proof}
Let \( d = k / 2\sp m \) and let \( U = \bigcup\sb { n \geq m} \bigcup\sb {s \in I\sb n} \Nbhd\sb s \) be a disjoint union, where each \( I\sb n \) is a (possibly empty) set of binary sequences of length \( n \).
So \( \mu ( U ) = \sum\sb { n = m }\sp {\infty } \card{ I\sb n } \cdot 2\sp {- n} \).
Let \( N \) be such that \( d \leq \sum\sb {n = m}\sp N \card{I\sb n } \cdot 2\sp {- n} \), and set \( J = \setof{ t \in \pre{ N}{2} }{ \exists s \in \bigcup\sb { n = m }\sp N I\sb n ( s\subseteq t )} \).
Then \( U = \bigcup\sb { t \in J } \Nbhd\sb t \cup \bigcup\sb { n > N } \bigcup\sb {s \in I\sb n} \Nbhd\sb s \).
Since \( d \leq \mu ( \bigcup\sb {t\in J} \Nbhd\sb t ) \), and \( \mu ( \Nbhd\sb t ) = 2\sp { - N } \) for every \( t \in J \), if \( J' \) is a subset of \( J \) of cardinality \( 2\sp {N-m}k \) , then \( V = \bigcup\sb { t \in J' } \Nbhd\sb t \) is as required.
\end{proof}

\begin{proof}[Proof of Proposition~\ref{prop:dualisticofallmeasures}]
Consider the following disjoint sets of nodes (see Figure~\ref{fig:whiteblack})
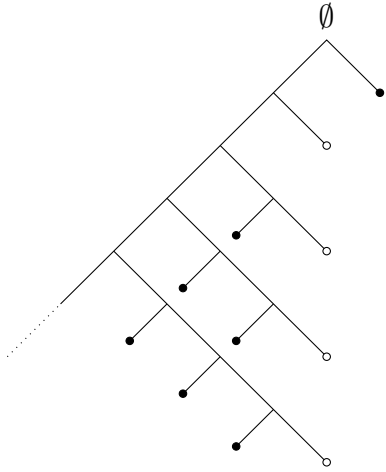
\begin{figure}
\begin{tikzpicture}[scale=0.7]
\draw (1,-5) -- (6,0)--(7,-1);
\draw[dotted] (0,-6)--(1,-5);
\draw (5,-1)--(6,-2);
\draw (4,-2)--(6,-4);
\draw (3,-3)--(6,-6);
\draw (2,-4)--(6,-8);
\draw (5,-3)--(4.3,-3.7);
\draw (4,-4)--(3.3,-4.7);
\draw (5,-5)--(4.3,-5.7);
\draw (3,-5)--(2.3,-5.7);
\draw (4,-6)--(3.3,-6.7);
\draw (5,-7)--(4.3,-7.7);
\filldraw (7, -1) circle (2pt);
\filldraw (4.3,-3.7) circle (2pt);
\filldraw (3.3,-4.7) circle (2pt);
\filldraw (4.3,-5.7) circle (2pt);
\filldraw (2.3,-5.7) circle (2pt);
\filldraw (3.3,-6.7) circle (2pt);
\filldraw (4.3,-7.7) circle (2pt);
\filldraw[fill=white] (6,-2) circle (2pt);
\filldraw[fill=white] (6,-4) circle (2pt);
\filldraw[fill=white] (6,-6) circle (2pt);
\filldraw[fill=white] (6,-8) circle (2pt);
\node at (6,0) [above] {\( \emptyset \)};
\end{tikzpicture}
 \caption{The nodes \( \circ \) are in \( A\sb 1 \), the nodes \( \bullet \) are in \( A\sb 2 \).}
 \label{fig:whiteblack}
\end{figure}
\begin{align*}
A\sb 1 & = \setofLR{ 0\sp {( n )} \conc 1\sp {( n )} }{ n > 0 }
\\
A\sb 2 & = \setofLR{ 0\sp {( n )} \conc 1\sp {( m )} \conc \seq{ 0 }}{ n > m > 0 } \cup \setLR{\seq{1} } .
\end{align*}
Then \( A\sb 1 \cup A\sb 2 \) is a maximal antichain, and \( \pre{\omega}{2} = U\sb 1 \cup U\sb 2 \cup \set{ 0\sp {( \omega ) } } \) is a partition, where \( U\sb i = \bigcup\sb { t \in A\sb i } \Nbhd\sb t \) (\( i = 1 , 2 \)) are open sets, and \( \mu ( U\sb 1 ) = 1 / 3 \) and \( \mu ( U\sb 2 ) = 2 / 3 \).
Then \( U\sb 1 = \bigcup\sb { n \geq 1} t\sb n \conc \pre{\omega}{2} \), where \( t\sb n = 0\sp {( n )} \conc 1\sp {( n )} \).
If we replace \( \pre{\omega}{2} \) with a smaller clopen set we obtain an open subset of \( U\sb 1 \).
To be more specific, fix clopen sets 
\[
D\sb 0 , \quad D\sb {1 / 4} , \quad D\sb {1 / 2} , \quad D\sb { 3 / 4} , \quad D\sb 1 
\]
of measure \( 0 \), \( 1/4 \), \( 1 / 2 \), \( 3 / 4 \), and \( 1 \) respectively.
(Thus \( D\sb 0 = \emptyset \), \( D\sb 1 = \pre{\omega}{2} \), while \( D\sb {1 / 4} \), \( D\sb { 1 / 2} \), \( D\sb { 3 / 4 } \) can be taken to be, for example, \( \Nbhd\sb {1\sp {( 2 )}} \), \( \Nbhd\sb {0} \) and \( \Nbhd\sb { 0 } \cup \Nbhd\sb {1\sp {( 2 )}} \).)
For each \( f \colon \omega \setminus \set{0} \to \set{ 0 , 1/4 , 1/2 , 3/4 , 1 } \) the set 
\[
W\sb f = \bigcup\sb { n \geq 1} t\sb n \conc D\sb {f ( n ) } \subseteq U\sb 1
\]
is open and since \( \Fr ( W\sb f ) \subseteq \set{ 0\sp {( \omega )} }\) and \( \density\sb {W\sb f } ( 0\sp { ( \omega ) } ) = 0 \), then \( W\sb f \) is dualistic.

\begin{claim}
For every \( r \leq 1 / 3 \) there is an \( f \) such that \( \mu ( W\sb f ) = r \).
\end{claim}

\begin{proof}[Proof of the Claim]
If \( r = 1 / 3 \) then take \( f \) such that \( f ( n ) = 1 \) for all \( n \), so that \( W\sb f = U\sb 1 \).
Therefore we may assume that \( r < 1 / 3 \).
Let \( \seqof{ u\sb n }{ n \geq 1 } \) be the \( 4 \)-ary expansion of \( r \), that is \( u\sb n \in \set{ 0 , 1 , 2 , 3 } \) and \( r = \sum\sb { n = 1 }\sp \infty u\sb n / 4\sp n \).
Let \( h \geq 1 \) be least such that \( u\sb h = 0 \) and \( u\sb n = 1 \) for all \( n < h \)---such \( h \) exists as \( 1 / 3 = \sum\sb { n = 1 }\sp \infty 1/ 4\sp n \).
Letting \( f ( n ) = 1 \) if \( n < h \), and \( f ( n ) = u\sb { n + 1 } / 4 \) if \( n \geq h \) we have that 
\[ 
\mu ( W\sb f ) = \sum\sb {n = 1}\sp \infty \frac{ f ( n ) }{ 2\sp { \lh t\sb n} }= \sum\sb {n = 1}\sp { h - 1 } \frac{1 }{ 4\sp n} + \sum\sb {n = h }\sp \infty \frac{u\sb {n + 1} / 4 }{ 4\sp n} = \sum\sb {n = 1}\sp { h - 1 } \frac{1 }{ 4\sp n} + \sum\sb {n = h + 1 }\sp \infty \frac{u\sb n }{ 4\sp n} = r . \qedhere
\]
\end{proof}

We can now finish the proof: given \( 0 < r < 1 \) we construct a dualistic open set \( U \) such that \( \mu ( U ) = r \), \( \Fr ( U ) \subseteq \setLR{ 0\sp {( \omega )} } \), and \( \density\sb U ( 0\sp {( \omega )} ) = 0 \).
If \( r \leq 1 / 3 \) then by the Claim we can take \( U = W\sb f \), so we may assume that \( 1 / 3 < r < 1 \).
Choose \( d \in \mathbb{D} \) such that \( 0 < r - 1 / 3 < d < 2 / 3 \) and by Lemma~\ref{lem:dyadicvalueofdualistic} let \( V \subseteq U\sb 2 \) be clopen of measure \( d \).
By the Claim there is a \( W\sb f \subseteq U\sb 1 \) of measure \( r - d < 1 / 3 \), so \( U = V \cup W\sb f \) has measure \( r \), and it is dualistic, since if \( x \in \Cl ( U ) \) then either there is a \( k \) such that \( x \restriction k \in A\sb 1 \cup A\sb 2 \), so that \( \density\sb U ( x ) \in \set{ 0 , 1} \), or else \( x = 0\sp {( \omega )} \) and \( \density\sb U ( x ) = \density\sb {W\sb f} ( 0\sp {( \omega )} ) + \density\sb {V} ( 0\sp {( \omega )} ) = 0 \).

The set \( U \) constructed above is open and dualistic; adding the point \( 0\sp {( \omega )} \) we obtain a closed set of the same measure which is still dualistic.
\end{proof}

We can now prove the following result.

\begin{theorem}\label{th:solidcountablerange}
For every countable \( S \subseteq ( 0 ; 1 ) \) there is a solid set \( A \subseteq \pre{\omega}{2} \), which can be either open or closed, such that \( \ran ( \density\sb A ) = S \cup \set{ 0 , 1 } \) and \( \FORALL{ r \in S }\EXISTSONE{z} ( \density\sb A ( z ) = r ) \).
\end{theorem}

\begin{proof}
If \( S = \emptyset \) let \( A \) be a nonempty, clopen proper subset of \( \pre{\omega}{2} \).
Suppose first \( S = \set{ r } \) is a singleton.
By Proposition~\ref{prop:dualisticofallmeasures} choose \( D\sp 1 \) open and \( D\sp 2 \) closed, both dualistic and such that \( \mu ( D\sp i ) = r \).
Then \( A\sp 1 = \bigcup\sb {n} 0\sp {( n )} \conc 1 \conc D\sp 1 \) is open, \( A\sp 2 = \set{0\sp {( \omega )}} \cup \bigcup\sb {n} 0\sp {( n )} \conc 1 \conc D\sp 2 \) is closed, \( \density\sb {A\sp i} ( 0\sp { ( \omega )} ) = r \), and they are as required.
Finally suppose \( S = \setof{ r\sb n }{ n < N } \) where \( 0 < N \leq \omega \).
By what we just proved, for each \( n < N \) let \( A\sb n\sp 1 \) be open and \( A\sb n\sp 2 \) be closed, satisfying the statement of the theorem for the set \( \set{ r\sb n } \) and such that \( \density\sb {A\sp i\sb n} ( 0\sp {( \omega ) } ) = r\sb n \).
The sets
\[
A\sp 1 = \bigcup\sb {n < N} 0\sp {( n )} \conc 1\sp {( n )} \conc A\sp 1\sb n , \qquad A\sp 2 = \set{0\sp { ( \omega ) } } \cup \bigcup\sb {n < N} 0\sp {( n )} \conc 1\sp {( n )} \conc A\sp 2\sb n
\]
are solid, since \( \density\sb {A\sp i} ( 0\sp {( \omega ) } ) = 0 \) and each \( A\sb n\sp i \) is solid, and \( \ran ( \density\sb {A\sp i} ) = S \), since \( \density\sb {A\sp i} ( 0\sp {( n )} \conc 1\sp {( n )} \conc 0\sp {( \omega ) } ) = r\sb n \) for all \( n \in \omega \).
Moreover \( A\sp 1 \) is open, and \( A\sp 2 \) is closed.
Observe that the uniqueness condition \( \FORALL{ r \in S }\EXISTSONE{z} ( \density\sb A ( z ) = r ) \) follows at once from the construction.
\end{proof}

\subsection{Stretching}\label{subsec:stretching} 
The \markdef{stretch} of \( s \in \pre{ \leq \omega }{2} \) is \( \overline{s} \in \pre{ \leq \omega }{2} \) defined by
\[
\overline{s} = s ( 0 ) \sp { ( 1 )} \conc s ( 1 ) \sp { ( 2 ) } \conc s ( 2 ) \sp { ( 3 ) } \conc \dots .
\]
Thus \( \overline{s} \in \pre{ < \omega }{2} \iff s \in \pre{ < \omega }{2} \).
Any \( X \subseteq \pre{ \omega }{2} \) and \( T \in \Tr\sb 2 \) can be stretched by letting
\[
\overline{X} = \setofLR{ \overline{x} }{ x \in X }, \qquad \overline{T} = \setof{ u \in \pre{ < \omega }{2}}{ \EXISTS{t \in T} ( u \subseteq \overline{t} ) } .
\]
Therefore \( \body{ \rule{1pt}{0pt}\overline{T}\rule{1pt}{0pt}} = \overline{ \body{T}} \) is null.
Every \( z \in \pre{\omega}{2} \setminus \overline{ \pre{\omega}{2} } \) has a largest (possibly empty) initial segment of the form \( \overline{s} \).

The main technical tool in this paper is the construction of continuous maps
\[
\PrTr\sb 2 \to \KK , \qquad T \mapsto K\sb T
\]
witnessing that the collection of all compact sets that have a specific property is \( \bGamma \)-hard for some pointclass \( \bGamma \).
The problem is that continuous reductions (a standard tool in descriptive set theory) do not preserve measure, so the space \( \pre{ \omega }{2} \) will be replaced by some homeomorphic copy \( C \) of measure zero.
The descriptive set theoretic issues are handled by \( C \), while its complement is where dualistic sets of positive measure are added in order for the construction to work.
We need special sequences flagging that we are reaching the complement of \( C \): the collection of all \markdef{flags of order \( n \)} is \( \FLAG ( n ) = \pre{ n + 1 }{2} \setminus \setLR{ 0\sp {( n + 1 )} , 1\sp {( n + 1 ) } } \), so that \( \FLAG ( 0 ) = \emptyset \).
For notational ease let us agree that 
\[
\flag \text{ denotes an element of \( \FLAG ( n ) \) for some } n .
\] 

Suppose we are given a pruned tree \( T \) on \( \set{ 0 , 1 } \) and that \( ( D\sb t )\sb {t \in T} \) is a \markdef{compliant sequence of sets}, meaning that \( D\sb t \) is dualistic, \( \emptyset \neq D\sb t \neq {}\sp {\omega } 2 \) and \( \mu ( \Int D\sb t ) = \mu ( \Cl D\sb t ) \).
(In our applications the \( D\sb t \)s will be either clopen sets, or else will be obtained via Proposition~\ref{prop:dualisticofallmeasures}, so compliance will never be an issue.)
Then 
\begin{align*}
K\sb T & = \overline{\body{T}} \cup \bigcup \setofLR{ \overline{t} \conc \flag \conc \Cl D\sb t } { t \in T \wedge \flag \in \FLAG ( \lh t )} ,
\\
O\sb T &= \bigcup \setofLR{ \overline{t} \conc \flag \conc \Int D\sb t } { t \in T \wedge \flag \in \FLAG ( \lh t )} ,
\end{align*}
are the \markdef{closed} and \markdef{open offspring of \( T \) determined by \( ( D\sb t )\sb {t \in T} \)}, respectively.
Note that \( K\sb T \) is closed, \( O\sb T \) is open, and that \( K\sb T =\sb \mu O\sb T \), so that the density function is the same for both sets.

\begin{lemma}\label{lem:averagealongabranch}
Let \( A \in \set{ O\sb T , K\sb T }\) be the open or closed offspring of \( T \) determined by a compliant \( ( D\sb t )\sb {t \in T} \). 
Let \( x \in \body{T} \), and let \( z = \overline{x} \).
Then \( \density\sb A ( z ) = \lim\sb { n \to \infty } \mu ( D\sb { x \restriction n } ) \) meaning that \( \density\sb A ( z ) \) is defined just in case \( \lim\sb { n \to \infty } \mu ( D\sb { x \restriction n } ) \) exists, and in that case they are equal.
\end{lemma}

\begin{proof}
Let \( \tau\sb k = k ( k + 1 ) / 2 \) be the \( k \)-th triangular number, so that \( \lh s = k \IFF \lh \overline{s} = \tau\sb k \).
By construction \( \card{\mu ( \LOC{A}{ z \restriction \tau\sb k } ) - \mu ( D\sb { x \restriction k} ) } \leq 2\sp {- k } \), so \( \mu ( \LOC{A}{ z \restriction \tau\sb k } ) \) converges iff \( \mu ( D\sb { x \restriction k} ) \) does, and in that case their limit is the same.
So if \( \lim\sb { k \to \infty } \mu ( D\sb { x \restriction k } ) \) does not exist then \( \oscillation\sb A ( z ) > 0 \).
Therefore it is enough to show that if \( \lim\sb { k \to \infty } \mu ( D\sb { x \restriction k } ) = r \) then \( \density\sb A ( z ) = \lim\sb {n \to \infty} \mu ( \LOC{A}{z \restriction n } ) = r \).
If \( \tau \sb k < n < \tau\sb { k + 1 } \), then \( z \restriction n = \overline{ x \restriction k} \conc i\sp {( m )} \) with \( 0 < m < k + 1 \) and \( i = x ( k ) \).
But \( \mu ( \LOC{A}{ z \restriction n } ) \) belongs to the closed interval with endpoints \( \mu ( D\sb { x \restriction k } ) \) and \( \mu ( \LOC{A}{ \overline{ x \restriction k + 1 } } ) = \mu ( \LOC{A}{z \restriction \tau\sb {k + 1 } } ) \), so the result follows from our assumptions.
\end{proof}

The next result is an immediate consequence of Lemma~\ref{lem:averagealongabranch} and the definition of offspring, and it is the blueprint for the main constructions in this paper.
It says that given any labelling \( \psi \) of a pruned tree \( T \) we can construct a closed/open offspring of \( T \) such that the behavior of its density function is completely determined by the value of the limit of \( \psi \) along the branches of \( T \).

\begin{theorem}\label{th:template}
Let \( T \in \PrTr\sb 2 \), let \( \psi \colon T \to ( 0 ; 1 ) \), and let \( ( D\sb t )\sb { t \in T } \) be a compliant sequence such that \( \mu ( \Int D\sb t ) = \mu ( \Cl D\sb t ) = \psi ( t ) \).
Let \( K\sb T \) and \( O\sb T \) be the closed and open offsprings of \( T \) generated by \( ( D\sb t )\sb { t \in T } \), and write \( \density \) for \( \density\sb {K\sb T} = \density\sb {O\sb T} \).
Then for all \( z \in \pre{\omega}{2} \) 
\begin{itemize}
\item
if \( z \notin \overline{\body{T}} \) then \( \density ( z )\in \setLR{ 0 , 1 } \),
\item
if \( z = \overline{x} \) with \( x \in \body{T} \) then \( \density ( z ) = \lim\sb { n \to \infty } \psi ( x \restriction n ) \) meaning that \( \density ( z ) \) is defined iff \( \lim\sb { n \to \infty } \psi ( x \restriction n ) \) exists, and in that case they are equal.
\end{itemize}
\end{theorem}

\section{Three constructions}\label{sec:constructions}
This is the main technical part of the paper, where we construct three maps from pruned trees on \( \setLR{ 0 , 1 } \) to compact subsets of \( \pre{\omega}{2} \).
The results of Sections~\ref{sec:mainresult} and~\ref{sec:projectivesubsetsofKK} are obtained by combining these maps.

\subsection{The first reduction}
For each continuous function \( c \colon \mathcal{N} \to [ 0 ; 1 ] \), the first reduction takes a tree \( T \in \PrTr\sb 2 \) and produces a compact set \( K \) such that each branch \( x \in \body{T} \cap \mathcal{N} \) corresponds to the point \( z = \overline{x} \) such that \( \density\sb K ( z ) = c ( x ) \), while all other exceptional points of \( K \) are blurry. 
Thus, if \( \ran ( c ) \subseteq ( 0 ; 1 ) \), then \( K \) is spongy exactly when \( T \) has no branches in \( \mathcal{N} \).

\begin{theorem}\label{th:firstreduction}
For every continuous \( c \colon\mathcal{N} \to [ 0 ; 1 ] \) there is a continuous function \( \firstreduction\sb c \colon \PrTr\sb 2 \to \KK \) such that for all \( T \in \PrTr\sb 2 \) the compact set \( \firstreduction\sb c ( T ) \) is the closed offspring of a compliant \( ( D\sb t )\sb { t \in T } \) and
\[ 
\FORALL{ x \in\body{T} } ( x \in\mathcal{N} \IFF \oscillation\sb { \firstreduction\sb c ( T ) } ( \overline{x} ) = 0) ,
\]
and whenever \( x \in \body{T} \cap \mathcal{N} \) then \( \density\sb { \firstreduction\sb c ( T ) } ( \overline{x} ) = c ( x ) \).
\end{theorem}

\begin{proof}
The idea is that while enumerating \( x \in \body{T} \subseteq \pre{\omega}{2} \), every time we reach a `\( 1 \)' we get one step closer to verifying that the density is \( c ( x ) \), while reaching a `\( 0 \)' will cause a small oscillation around the current approximation of \( c ( x ) \).

Let \( \varphi\sp - , \varphi\sp + \colon \pre{ < \omega}{2} \to \mathbb{D} \) be dyadic approximations of \( c \) (see Section~\ref{subsubsec:continuousfunctionsfromBaire}) such that \( \varphi\sp - ( t ) < \varphi\sp + ( t ) \) and \( \varphi\sp + ( t ) - \varphi\sp - ( t ) < 2\sp { - \lh t } \) for all \( t \in \pre{ < \omega }{2} \).
Let \( \psi \colon T \to \mathbb{D} \) 
\[
 \psi ( t ) = 
 	\begin{cases}
	\varphi\sp - ( \hd t ) & \text{if \( \lh ( \zt t ) \) is even,}
	\\
	\varphi\sp + ( \hd t ) & \text{otherwise.}
	\end{cases}
\]
(The functions \( \hd \) and \( \zt \) are defined in~\eqref{eq:head} and~\eqref{eq:tail}.)
For each \( t \in T \), choose the canonical clopen set \( D\sb t \) of measure \( \psi ( t ) \).
By clopennes \( ( D\sb t )\sb { t \in T} \) is compliant.
Then \( T \mapsto \firstreduction\sb c ( T ) \) is continuous, where \( \firstreduction\sb c ( T ) \) is the closed offspring of \( T \) generated by \( ( D\sb t )\sb { t \in T} \).

Fix \( x \in \body{T} \).
If \( x \in \mathcal{N} \) then \( \lim\sb {n \to \infty} \psi ( x \restriction n ) = c ( x ) \).
If \( x \notin \mathcal{N} \) let \( M \) be least such that \( x ( k ) = 0 \) for all \( k \geq M \): then \( \psi ( x \restriction M + 2 k ) = \varphi \sp - ( x \restriction M ) \) and \( \psi ( x \restriction M + 2 k + 1 ) = \varphi \sp + ( x \restriction M ) \) for all \( k \), so \( ( \psi ( x \restriction n ) )\sb n \) does not converge.
Therefore we are done by Theorem~\ref{th:template}.
\end{proof}

\subsection{The second reduction}
The second reduction takes a tree \( T \in \PrTr\sb 2 \) and produces a compact set \( K \) such that any branch \( x \in \body{T} \cap \mathcal{N} \) corresponds to the point \( z = \overline{z}\) such that \( \oscillation\sb K ( z ) = 1 \), and in all other points the density is either \( 0 \) or \( 1 \); thus \( K \) is quasi-dualistic, and it is dualistic (and hence solid) exactly when \( T \) has no branches in \( \mathcal{N} \).

\begin{theorem}\label{th:secondreduction}
There is a continuous \( \secondreduction \colon \PrTr\sb 2 \to \KK \) such that \( \secondreduction ( T ) \) is a closed offspring of \( T \), and for all \( x \in \body{T} \)
\begin{align*}
x \in \mathcal{N} & \IMPLIES \oscillation\sb { \secondreduction ( T ) } (\overline{x} ) = 1 ,
 \\
x \notin \mathcal{N} & \IMPLIES \density\sb { \secondreduction ( T ) } ( \overline{x} ) \in \set{ 0 , 1} .
\end{align*}
\end{theorem}

\begin{proof}
Let \(\varphi \colon T \to \mathbb{D} \) be the map 
\[
 \varphi ( t ) = \begin{cases}
 1- 2\sp { - ( 1 + \lh t )} & \text{if \( t \) is even,}
 \\
 2\sp { - ( 1 + \lh t )} & \text{if \( t \) is odd,}
 \end{cases}
\]
and choose \( D\sb t \) clopen so that \( \mu ( D\sb t ) = \varphi ( t ) \).
Then \( ( D\sb t )\sb { t \in T } \) is compliant, and let \( \secondreduction ( T ) \) be the closed offspring of \( T \) generated by \( ( D\sb t )\sb { t \in T } \).
The map \( T \mapsto \secondreduction ( T ) = K \) is continuous, and given \( x \in \body{T} \)
\begin{align*}
x \in \mathcal{N} & \IMPLIES \oscillation\sb K ( \overline{x} ) = 1 ,
\\
x\notin \mathcal{N} \text{ is even} & \IMPLIES \density\sb K ( \overline{x} ) = 1 ,
\\
x\notin \mathcal{N} \text{ is odd} & \IMPLIES \density\sb K ( \overline{x} ) = 0 . \qedhere
\end{align*}
\end{proof}

\subsection{The third reduction}
The third reduction takes a tree \( T \in \PrTr\sb 2 \) and produces a compact set \( K \) such that if \( \body{T} \cap \mathcal{N} \neq \emptyset \) then \( ( 0 ; 1 ) \cap \ran \density\sb K = \ran c \) where \( c \) is some continuous function chosen in advance, otherwise \( K \) is spongy.

\begin{theorem}\label{th:thirdreduction}
If \( U \) is a pruned tree on \( \omega \) and \( c \colon \body{U} \to ( 0 ; 1 ) \) is continuous, then there is a continuous function \( \thirdreduction\sb c \colon \PrTr\sb 2 \to \KK \) such that 
\begin{align*}
T \in \WF\sb 2 & \IMPLIES \thirdreduction\sb c ( T ) \in \Spongy ,
\\
T \in \IF\sb 2 & \IMPLIES \ran ( \density\sb {\thirdreduction\sb c ( T )} ) = \ran ( c ) \cup \set{ 0, 1 }.
\end{align*}
\end{theorem}

\begin{proof}
By Lemma~\ref{lem:tight} we may assume that \( U \) is gapless and \( c \) is Lipschitz.
Using the homeomorphism \( \boldsymbol{h} \colon \pre{\omega}{\omega} \to \mathcal{N} \) from~\eqref{eq:BaireinCantor0} we can turn \( U \) into a pruned tree \( V \) on \( 2 \).
More precisely, let \( V\in \PrTr\sb 2 \) be such that \( \body{V} = \Cl ( \appl{\boldsymbol{h}}{ \body{U} } ) \), so that \( \body{V} \cap \mathcal{N} = \appl{\boldsymbol{h}}{ \body{U} } \).
The compact set \( K = \thirdreduction\sb c ( T ) \) is obtained as a closed offspring of the pruned tree 
\[
T \oplus V \equalsdef { \downarrow} \setof{ t \oplus v }{ t \in T \wedge v \in V \wedge \lh t = \lh v } ,
\]
with \( t \oplus v \) as in~\eqref{eq:oplus}.
As \( \overline{\body{T \oplus V}} = \setof{ \overline{x \oplus y} }{ x \in \body{T} \AND y \in \body{V}} \), it is enough to guarantee that for all \( x \in \body{T} \) and \( y \in \body{V} \)
\begin{align}
 x \notin \mathcal{N} \vee y \notin \mathcal N & \IMPLIES \oscillation\sb K ( \overline{x \oplus y} ) > 0 , \label{eq:thirdreduction-1}
\\
x \in \mathcal{N} \wedge y \in \mathcal{N} & \IMPLIES \density\sb K ( \overline{x \oplus y} ) = c \bigl ( \boldsymbol h\sp {-1} ( y ) \bigr ). \label{eq:thirdreduction-2}
\end{align}

\begin{claim}
There exists \( \varphi \colon U \to \mathbb{D} \) a dyadic approximation of \( c \) such that for all \( u \in U \)
\begin{gather}
\FORALL{k \in \omega } ( u \conc \seq{ k } \in U ) \IMPLIES \lim\sb { k \to \infty} \varphi ( u \conc \seq{ k } ) \text{ does not exist} , \label{eq:thirdreduction-4}
\\
\shortintertext{and}
 \FORALL{k , h \in \omega } ( u \conc \seq{ k } , u \conc \seq{ h } \in U \IMPLIES \card{ \varphi ( u \conc \seq{ k } ) - \varphi ( u \conc \seq{ h } ) } < 2\sp { 1 - \lh u } ) . \label{eq:thirdreduction-5}
\end{gather}
\end{claim} 

\begin{proof}[Proof of the Claim]
To get such \( \varphi \), start with the canonical dyadic approximation \( \psi \colon U \to \mathbb{D} \) of \( c \). 
For every \( u \in U \) such that \( \FORALL{k \in \omega } ( u \conc \seq{ k } \in U ) \), choose \( d\sb {k , u} \in \mathbb{D} \) so that \( k \mapsto d\sb {k , u} \) does not converge and \( \card{ d\sb {k , u} - \psi ( u \conc \seq{ k } ) } < 2\sp {- ( 1 + \lh u )} \).
Set 
\[
\varphi ( v ) = \begin{cases}
	d\sb {k , u} & \text{if \( v = u \conc \seq{ k } \) and } \FORALL{ j } ( u \conc \seq{ j } \in U ),
	\\
	\psi ( v ) & \text{otherwise.}
 	\end{cases}
\]
The choice of the \( d\sb {k , u} \) guarantees that~\eqref{eq:thirdreduction-4} holds.
Suppose \( u \conc \seq{ h } , u \conc \seq{ k } \in U \): by~\eqref{eq:errorindyadicapproximation} we have that \( \card{ \psi ( u \conc \seq{ h } ) - \psi ( u \conc \seq{ k } ) } < 2\sp {- \lh u} \), and as \( \FORALL{v \in U} ( \card{ \varphi ( v ) - \psi ( v ) } < 2\sp {- \lh v } ) \), then \( \card{ \varphi ( u \conc \seq{ h } ) - \varphi ( u \conc \seq{ k } ) } < 2\sp { 1 - \lh u} \), and hence~\eqref{eq:thirdreduction-5} is satisfied.
It follows that \( \varphi \) is the desired approximation.
\end{proof}

We must define the clopen sets \( D\sb s \) for \( s \in T \oplus V \); as usual it is enough to specify a dyadic value for \( \mu ( D\sb s ) \).
Fix \( \varphi\sp - , \varphi\sp + \colon U \to \mathbb{D} \) such that \( \varphi\sp - ( u ) < \varphi ( u ) < \varphi\sp + ( u ) \) and \( \varphi\sp + ( u ) - \varphi\sp - ( u ) < 2\sp { - \lh u } \) for all \( u \in U \).
\begin{description}
\item[Case 1] 
\( s = t \oplus v \) where \( t \in T \), \( v \in V \). 
Then
\[
\mu ( D\sb s ) = \begin{cases}
 \varphi \sp + \left ( \widehat{ v \restriction \LENGTH ( t ) } \right ) & \text{if \( \zt ( t ) \) is even,}
 \\
 \varphi \sp - \left ( \widehat{ v \restriction \LENGTH ( t ) } \right ) & \text{if \( \zt ( t ) \) is odd.}
\end{cases}
\]
\item[Case 2] 
\( s = ( t \oplus v ) \conc i \) where \( t \in T \), \( v \in V \), and \( i \in 2 \).
Then 
\[
\mu ( D\sb s ) = \varphi ( \hat{w } ) , \text{ where } w = ( v \restriction \LENGTH ( t ) ) \conc 1.
\] 
\end{description}
Given \( x \in \body{T} \) and \( y \in \body{V} \) we have three possibilities:
\begin{itemize}
\item
If \( x \notin \mathcal{N} \) then \( x \oplus y = ( t \conc 0\sp {( \omega ) }) \oplus y \), so~\eqref{eq:thirdreduction-1} holds by Case~1.
\item 
If \( x \in \mathcal{N} \), \( y \notin \mathcal{N} \), let \( v = \hd ( y ) \).
Let \( \seqof{ t\sb h }{ h \in \omega } \) be the sequence of all restrictions of \( x \) ending with a \( 1 \) and such that \( \LENGTH ( t\sb h ) \geq \lh ( v ) \), and set \( s\sb h = ( t\sb h \oplus (y \restriction \lh ( t\sb h ) ) ) \conc x ( \lh ( t\sb h ) ) \) and \( v\sb h = v \conc 0\sp {( h )} \conc \seq{1} \).
Then \( \mu ( D\sb {s\sb h} ) = \varphi \left ( \widehat{ v\sb h } \right ) = \varphi \left ( \hat{v} \conc \seq{ h } \right ) \) does not converge by~\eqref{eq:thirdreduction-4}, so~\eqref{eq:thirdreduction-1} holds.
\item 
Suppose \( x , y \in \mathcal{N} \) and fix \( \varepsilon > 0 \).
By~\eqref{eq:thirdreduction-5} there is \( k\sb 0 \in\omega \) such that for every \( k\geq k\sb 0 \) and every \( m \) one has \( \card{ \varphi ( ( \boldsymbol h\sp {-1} ( y ) \restriction k) \conc \seq{ m } ) - c \left ( \boldsymbol{h}\sp {-1} ( y ) \right ) } < \varepsilon \).
Consequently, for every such \( k \) and every \( n \) such that \( \LENGTH ( y \restriction \LENGTH ( x \restriction n ) ) > k \),
\[
\begin{split}
 \card{ \mu ( D\sb { ( x \restriction n ) \oplus ( y \restriction n) } ) - c \left ( \boldsymbol{h}\sp {-1} ( y ) \right ) } & \leq \card{ \varphi \sp \pm ( \widehat{ y \restriction \LENGTH ( x \restriction n ) }) - \varphi ( \widehat{ y \restriction \LENGTH ( x \restriction n ) }) } 
\\
& \qquad\qquad {}+ \card{ \varphi ( \widehat{ y \restriction \LENGTH ( x \restriction n ) }) - c \left ( \boldsymbol{h}\sp {-1} ( y ) \right ) }
\\
 & < 2\sp {- \lh ( \widehat{ y \restriction \LENGTH ( x \restriction n ) } )} + \varepsilon 
\\
 & = 2\sp {- \LENGTH ( y \restriction \LENGTH ( x \restriction n ) )} + \varepsilon 
\end{split}
\]
and since \( \LENGTH ( x \restriction n ) \geq \LENGTH ( y \restriction \LENGTH ( x \restriction n ) ) \) one has that
\[
 \card{ \mu (D\sb { ( ( x \restriction n ) \oplus ( y \restriction n ) ) \conc x ( n ) } ) - c \left (\boldsymbol{h}\sp {-1} ( y ) \right )} < \varepsilon .
\]
Then \( \lim\sb {n \to \infty} \mu ( D\sb { ( x \oplus y ) \restriction n } ) = c \left ( \boldsymbol{h}\sp {-1} ( y ) \right ) \) and therefore~\eqref{eq:thirdreduction-2} holds. \qedhere
\end{itemize}
\end{proof}

\section{The main result}\label{sec:mainresult}
Which analytic \( S \subseteq [0, 1] \) are of the form \( \ran \density\sb A \) for some measurable set \( A \)?
Both \( \set{0} \) and \( \set{1} \) are of this form---just take \( A = \emptyset \) and \( A = \pre{\omega}{2} \), respectively. 
If \( x \in ( 0 ; 1 ) \) belongs to some \( \ran \density\sb A \), then \( 0 < \mu ( A ) < 1 \), so both \( 0 \) and \( 1 \) belong to \( \ran \density\sb A \).
This implies that if \( \emptyset \neq S \subseteq ( 0 ; 1 ) \) is analytic, then none of \( S \), \( S \cup \set{0} \), \( S \cup \set{1} \) can be the range of a density function.
The next result yields a complete answer to the question at the beginning of this section.

\begin{theorem}\label{th:Sigma11solid}
For any analytic \( S \subseteq ( 0 ; 1 ) \) there is a solid set \( A \subseteq \pre{\omega}{2} \), which can be taken to be either closed or open, such that \( \ran \density\sb A = \setLR{ 0 , 1} \cup S \).
\end{theorem}

\begin{proof}
If \( S \) is empty we can take \( A \) to be clopen, so we may assume that \( S \neq \emptyset \).
First we prove the result for \( A \) a closed set.

Let \( c \colon \body{U} \to ( 0 ; 1 ) \) be continuous and such that \( \ran c = S \).
By Lemma~\ref{lem:tight} we may assume that \( c \) is Lipschitz with \( U \) a pruned gapless tree on \( \omega \), \emph{e.g.}, \( U = \pre{ < \omega }{ \omega } \).
For every \( u \in U \) let \( i\sb u = \inf \setof{ c ( x ) }{ x \in \Nbhd\sb u } \) and \( s\sb u = \sup \setof{ c ( x ) }{ x \in \Nbhd\sb u } \) and set
\begin{equation}\label{eq:th:Sigma11solid-0}
J\sb u = \begin{cases}
\left ( i\sb u ; s\sb u \right ) & \text{if } i\sb u < s\sb u ,
\\
\setLR{ i\sb u } &\text{if } i\sb u = s\sb u .
\end{cases} 
\end{equation}
Therefore the length of \( J\sb u \) is \( \leq 2\sp {- \lh ( u ) } \).
Let \( Q = \setof{ q\sb n }{ n \in \omega } \subseteq S \) be such that \( \FORALL{u \in U} ( Q \cap J\sb u \neq \emptyset ) \).
Then \( Q \) is dense in \( S \), and let \( \psi \colon U \to Q \) be defined by
\[
 \psi ( u ) = q\sb k \text{ where \( k \) is least such that } q\sb k \in J\sb u .
\]
Then \( \psi \) is a \( Q \)-approximation of \( c \), and let \( T = {\downarrow} \setof{ \check{u} }{ u \in U } \), so that the function \( \boldsymbol{h } \) of~\eqref{eq:BaireinCantor0} maps homeomorphically \( \body{U} \) onto \( \body{T} \cap \mathcal{N } \).
Let \( \varphi \colon T \to Q \) be defined by \( \varphi ( \check{u} \conc 0\sp {( n )} ) = \psi ( u ) \).

 We will construct dualistic (albeit not necessarily clopen) sets \( D\sb t \) for all \( t \in T \) so that \( A \), the closed offspring of \( T \) determined by \( ( D\sb t )\sb { t \in T } \), satisfies
\begin{subequations}
\begin{align}
z \in \body{ T } \cap \mathcal{N} & \IMPLIES \density\sb A ( \overline{z} ) = c ( \boldsymbol{h }\sp {-1} ( z ) ) \label{eq:th:Sigma11solid-1}
\\ 
z \in \body{ T } \setminus \mathcal{N} & \IMPLIES \density\sb A ( \overline{z} ) \in Q . \label{eq:th:Sigma11solid-2}
\end{align}
\end{subequations}
As \( D\sb t \in \Dual \) then \( \density\sb A ( z ) \in \set{ 0 , 1} \) for all \( z \notin \overline{\body{T} } \), and~\eqref{eq:th:Sigma11solid-1} implies that \( S \subseteq ( 0 ; 1 ) \cap \ran \density\sb A \) while~\eqref{eq:th:Sigma11solid-2} yields the other inclusion.
Therefore \( A \) is solid and \( ( 0 ; 1 ) \cap \ran \density\sb A = S \).

Thus it is enough to define the sets \( D\sb t \).
As every node \( t \in T \) is of the form \( \check{u} \conc 0\sp { ( n ) } \) with \( u \in U \), use Proposition~\ref{prop:dualisticofallmeasures} to choose \( D\sb t \) of measure \( \varphi ( t ) = \psi ( u ) \).
Suppose \( z \in \body{T} \): by Lemma~\ref{lem:averagealongabranch} \( \density\sb A ( \overline{z} ) = \lim\sb {n \to \infty} \mu ( D\sb { z \restriction n } ) = \lim\sb {n \to \infty} \varphi ( z \restriction n ) \).
If \( z \in \mathcal{N} \), then letting \( x = \boldsymbol{h}\sp {-1} ( z ) \), we have that \( \density\sb A ( \overline{z} ) = \lim\sb {n \to \infty} \psi ( x \restriction n ) = c ( x ) \), so~\eqref{eq:th:Sigma11solid-1} holds.
If \( z \notin \mathcal{N} \) then \( \varphi ( z \restriction n ) \) is constantly equal to some \( q\sb k \) for \( n \) sufficiently large, so~\eqref{eq:th:Sigma11solid-2} holds.

As the values of \( \ran ( c ) \) are attained exactly on the frontier of \( A \), then \( \Int A \) is open and such that \( \ran ( \density\sb {\Int A} ) = S \cup \set{ 0 , 1 } \).
\end{proof}

\begin{lemma}\label{lem:Borelsolid}
If \( B \) is an uncountable Borel subset of a Polish space \( X \), then there is a partition \( B = P \cup Q \) such that
\begin{itemize}
\item
\( Q \) countable, 
\item
\( P \subseteq \Cl Q \), and 
\item
\( P \) is the continuous injective image of \( \body{U} \), with \( U \) a perfect tree on \( \omega \).
\end{itemize}
\end{lemma}

\begin{proof}
Let \( \tau \) be the topology on \( X \), and let \( \tau \sb 0 \) be a Polish topology extending \( \tau \) such that \( B \) is \( \tau \sb 0 \)-closed.
By Cantor-Bendixson there is a countable set \( C\sb 0 \) such that \( P\sb 0 = B \setminus C\sb 0 \) is closed and perfect with respect to \( \tau\sb 0 \).
Let \( D \subseteq P\sb 0 \) be countable and \( \tau \)-dense in \( P\sb 0 \), and let \( P\sb 1 = P\sb 0 \setminus D \).
Let \( \tau \sb 1 \) be a zero-dimensional Polish topology extending \( \tau\sb 0 \) such that \( P\sb 1 \) is \( \tau \sb 1 \)-closed.
By Cantor-Bendixson again \( P\sb 1 \) can be partitioned as \( P \cup C\sb 1 \) with \( C\sb 1 \) countable and \( P \) closed and perfect with respect to \( \tau\sb 1 \).
Set \( Q = C\sb 0 \cup D \cup C\sb 1 \).
Then \( P \) is \( \tau \sb 1 \)-homeomorphic to \( \body{U} \), for some \( U \) a perfect tree on \( \omega \), and hence it is the \( \tau \)-continuous injective image of \( \body{U} \).
\end{proof}

\begin{theorem}\label{th:Borelsolid}
For any Borel \( B \subseteq ( 0 ; 1 ) \) there is a solid set \( A \in \Range\sb {\mathrm{inj}} \), which can be taken to be either closed or open, such that \( \ran \density\sb A = \setLR{ 0 , 1} \cup B \).
\end{theorem}

\begin{proof}
As in Theorem~\ref{th:Sigma11solid} it is enough to prove the result when \( A \) is closed.
If \( B \) were countable, the result would follow from Theorem~\ref{th:solidcountablerange}, so we may assume that \( B \) is uncountable.
By Lemma~\ref{lem:Borelsolid} \( B = P \cup Q \) with \( c \colon \body{U} \to P \) a continuous bijection and \( U \) a pruned tree on \( \omega \).
Fix \( \seqof{ u\sb n }{ n \in \omega } \) and \( \seqof{ q\sb n }{ n \in \omega } \) be enumerations without repetitions of \( U \) and \( Q \), respectively. 
Arguing as in the proof of Theorem~\ref{th:Sigma11solid} let \( T = { \downarrow} \setof{ \check{u} }{ u \in U } \in \PrTr\sb 2 \) and for \( u \in U \) let \( J\sb u \subseteq ( 0 ; 1 ) \) be as in~\eqref{eq:th:Sigma11solid-0}.
Define \( \varphi \colon U \to Q \) by induction on \( n \) as follows: 
\[
 \varphi ( u\sb n ) = q\sb k \text{ where \( k \) is least such that } q\sb k \in J\sb {u\sb n} \setminus \setof{ \varphi ( u\sb m) }{ m < n} .
\]
This is well-defined since, as \( P \) has no isolated points, then \( J\sb {u\sb n} = ( i\sb {u\sb n } , s\sb {u\sb n } ) \) with \( i\sb {u\sb n } < s\sb {u\sb n } \), and so \( ( J\sb {u\sb n} \setminus \setof{ \varphi ( u\sb m) }{ m < n} ) \cap P \) is a nonempty, relatively open subset of \( P \).
As in Theorem~\ref{th:Sigma11solid} construct dualistic sets \( D\sb t \) with \( t \in T \) such that \( \mu ( D\sb t ) = \varphi ( u ) \), where \( \check{u} = \hd ( t ) \).
Let \( K\sb 0 \) be the closed offspring of \( T \) generated by \( ( D\sb t )\sb { t \in T} \): if \( z \in \body{T} \cap \mathcal{N} \) then \( \density\sb {K\sb 0} ( \overline{z} ) = c ( \boldsymbol{ h }\sp {-1} ( z ) ) \), and if \( z \in \body{T} \setminus \mathcal{N} \) then \( \density\sb {K\sb 0} ( \overline{z} ) \in Q \).
If \( z , w \in \body{T} \setminus \mathcal{N} \) are distinct, then \( u\sb n = \hd ( z ) \) and \( u\sb m = \hd ( w ) \) are distinct, and therefore \( \density\sb {K\sb 0} ( \overline{z} ) = \varphi ( u\sb n ) \neq \varphi ( u\sb m ) = \density\sb {K\sb 0} ( \overline{w} ) \).
Therefore by construction \( K\sb 0 \in {\KK} \cap \Range\sb {\mathrm{inj}} \) and
\[
P \subseteq \ran \density\sb { K\sb 0} \subseteq P \cup Q .
\]
The set \( S = B \setminus \ran \density\sb { K\sb 0} \) is countable, so by Theorem~\ref{th:solidcountablerange} there is \( K\sb 1 \in \KK \cap \Range\sb {\mathrm{inj}} \) such that \( \ran \density\sb { K\sb 1 } = S \).
Then \( K = 0 \conc K\sb 0 \cup 1 \conc K\sb 1 \) is the compact set we were looking for.
\end{proof}

By putting together the previous results we have:

\begin{theorem}\label{th:universalSigma11}
Let \( \mathcal{U} = \setofLR{ ( K , r ) \in \KK \times ( 0 ; 1 ) }{ \EXISTS{ z \in \pre{ \omega }{2} } ( \density\sb K ( z ) = r ) } \), and let \( \mathcal{U}\sp \complement = \left ( \KK \times ( 0 ; 1 ) \right ) \setminus \mathcal{U} \).
Then
\begin{enumerate-(a)}
\item
\( \mathcal{U} \) is \( \varSigma\sp 1\sb 1 \) and universal for \( \bSigma\sp {1}\sb {1} \) subsets of \( ( 0 ; 1 ) \), and therefore \( \mathcal{U}\sp \complement \) is \( \varPi\sp 1\sb 1 \) and universal for \( \bPi\sp {1}\sb {1} \) subsets of \( ( 0 ; 1 ) \).
\item
\( \mathcal{U} \cap \left ( \Solid \times ( 0 ; 1) \right ) \) parametrizes the \( \bSigma\sp {1}\sb {1} \) subsets of \( ( 0 ; 1 ) \), that is every analytic set \( S \subseteq ( 0 ; 1 ) \) is of the form \( \vsection{\mathcal{U}}{K} \) with \( K \) solid.
Similarly \( \mathcal{U}\sp \complement \cap \left ( \Solid \times ( 0 ; 1) \right ) \) parametrizes the \( \bPi\sp {1}\sb {1} \) subsets of \( ( 0 ; 1 ) \).
\item
\( \mathcal{U} \cap \left ( ( \Solid \cap \Range\sb {\mathrm{inj}} ) \times ( 0 ; 1) \right ) \) parametrizes the \( \bDelta\sp {1}\sb {1} \) subsets of \( ( 0 ; 1 ) \), and therefore \( \mathcal{U}\sp \complement \cap \left ( ( \Solid \cap \Range\sb {\mathrm{inj}} ) \times ( 0 ; 1) \right ) \) is \( \varPi\sp 1\sb 1 \) and parametrizes the \( \bDelta\sp {1}\sb {1} \) subsets of \( ( 0 ; 1 ) \).
\end{enumerate-(a)}
\end{theorem}

\begin{theorem}\label{th:uniformity}
If \( S \subseteq ( 0 ; 1 ) \) is an uncountable analytic set, then \( \setof{ K \in \KK }{ \ran \density\sb K = \set{ 0 , 1 } \cup S } \) is \( \bPi\sp {1}\sb {2} \)-complete in \( \KK \).

In particular \( \setof{ K \in \KK }{ \ran \density\sb K = [ 0 ; 1 ] } \) is \( \bPi\sp {1}\sb {2} \)-complete.
\end{theorem}

\begin{proof}
Let \( C \subseteq S \) be a closed perfect set. 
By Theorem~\ref{th:Sigma11solid} fix a compact set \( H \) such that \( \ran ( \density\sb H ) = \set{ 0 , 1 } \cup S \setminus C \).
For \( P \subseteq \pre{ \omega }{2} \) a \( \bPi\sp {1}\sb {2} \) set we construct a continuous map \( F \colon \pre{\omega}{2} \to \KK \) so that 
\begin{itemize}
\item
\( \ran ( \density\sb {F ( z ) } ) \subseteq \set{ 0 , 1 } \cup C \) for all \( z \in \pre{\omega}{2} \), and
\item
\( z \in P \iff \ran ( \density\sb {F ( z ) } ) = \set{ 0 , 1 } \cup C \).
\end{itemize}
Therefore \( z \mapsto 0 \conc F ( z ) \cup 1 \conc H \) witnesses that 
\[ 
P \leqW \setof{ K \in \KK }{ \ran ( \density\sb K ) = \set{ 0 , 1 } \cup S } ,
\] 
which is what we have to prove.

So fix \( P = \setof{ z \in \pre{ \omega }{2} }{ \FORALL{ y \in \mathcal{N} } \left ( ( y , z ) \in A \right )} \) with \( A \subseteq \mathcal{N} \times \pre{ \omega }{2} \) a \( \bSigma\sp {1}\sb {1} \) set.
\begin{claim}
There is \( T \in \PrTr \sb { 2 \times 2 \times 2 } \) such that 
\[ 
A = \setof{( y , z ) \in \mathcal{N} \times \pre{ \omega }{2} }{ \EXISTS{ x \in \mathcal{N}} \left ( ( x , y , z ) \in \body{T} \right ) } 
\]
and \( T ( z ) \) is pruned, for all \( z \in \pre{\omega}{2} \).
\end{claim}

\begin{proof}
Fix \( U \in \PrTr \sb { 2 \times 2 \times 2 } \) such that \( A = \setof{( y , z ) }{ \EXISTS{ x \in \mathcal{N}} \left ( ( x , y , z ) \in \body{U} \right ) } \), and let 
\begin{multline*}
T = U \cup \setl{ ( u \conc 0\sp {(k)} , v\conc 0\sp {(k)} , z \restriction k + \lh u ) }{ \lh u = \lh v \wedge z \in \pre{\omega}{2} \wedge k \in \omega}{ {} \wedge ( u , v ) \text{ is a terminal node of } T ( z ) }
\\
\setr{ ( u \conc 0\sp {(k)} , v\conc 0\sp {(k)} , z \restriction h + k ) }{ h = \lh u = \lh v \wedge z \in \pre{\omega}{2} \wedge k \in \omega}{ {} \wedge ( u , v ) \text{ is a terminal node of } U ( z ) } . \qedhere
\end{multline*}
\end{proof}

Fix \( c \colon \pre{\omega}{\omega} \to C \) a continuous bijection~\cite[Exercise 7.15]{Kechris:1995kc}, and let \( \thirdreduction = \thirdreduction\sb c \) be as in Theorem~\ref{th:thirdreduction}.
Then \( \thirdreduction ( {}\sp { < \omega }2 ) \) is the largest offspring that can be constructed using \( \thirdreduction \). 
The function \( G \colon \PrTr\sb 2\to \KK \) 
\[
G ( U ) = \thirdreduction ( {}\sp { < \omega }2 ) \setminus \bigcup\sb { t \in {}\sp { < \omega } 2 \setminus U} \Nbhd\sb { \bar t }
\]
is continuous.
Indeed, if \( U , U' \) coincide on all sequences of length at most \( n \), then for every \( z \in G ( U ) \) there is \( z' \in G ( U' ) \) with \( z \restriction \tau\sb n = z' \restriction \tau\sb n \), where \( \tau\sb n = n ( n + 1) / 2 \).
Notice also that \( G ( U ) \) is a closed offspring of \( U \).

Let now \( T \) be as in the Claim and let \( f \colon {}\sp {\omega }2\to \PrTr\sb 2 \) be the continuous function defined by
\[
 f ( z ) = {\downarrow} \setof{ u \oplus v }{ ( u , v ) \in T ( z ) } .
\]
Set \( F = G \circ f \).
By the definition of \( \thirdreduction \) and Theorem~\ref{th:template} applied to the tree \( f ( z ) \), one has that \( \ran ( \density\sb { F ( z ) } ) \subseteq \set{ 0 , 1 } \cup C \) for every \( z \in {}\sp {\omega }2 \).
Moreover, again from the definition of \( \thirdreduction \) and Theorem~\ref{th:template}
\begin{align*}
z \in P & \iff \FORALL{y \in \mathcal{N}} ( y , z ) \in A \notag
\\
 & \iff \FORALL{y \in \mathcal{N}} \EXISTS{ x \in \mathcal{N}} \left ( ( x , y , z ) \in \body{ T} \right ) 
 \\
 & \iff \FORALL{y \in \mathcal{N}} \EXISTS{ x \in \mathcal{N}} \left ( ( x , y ) \in \body{ T( z ) } \right ) 
 \\
 & \iff \FORALL{ y \in \mathcal{N} } \EXISTS{ x \in \mathcal{N}} \left ( x \oplus y \in \body{ f ( z ) } \right ) 
 \\
 & \iff \FORALL{ y \in \mathcal{N} } \EXISTS{ x \in \mathcal{N}} \left ( \density\sb {F( z ) } ( \overline{ x \oplus y } ) = c ( \boldsymbol h\sp { - 1 } ( y ) ) \right ) 
 \\
 & \iff \FORALL{ y \in \mathcal{N} } \left ( c ( \boldsymbol h\sp {-1} ( y ) ) \in \ran ( \density\sb {F ( z ) } ) \right ) & \text{(as \( c \) is injective)}
 \\
 & \iff \ran ( \density\sb { F ( z ) } ) = \set{ 0 , 1 } \cup C . 
\end{align*}
So \( F \) is the desired reduction.
\end{proof}

\section{Projective subsets of \texorpdfstring{\( \KK \)}{K} and \texorpdfstring{\( \MALG \)}{MALG}} \label{sec:projectivesubsetsofKK}
We analyze the descriptive set theoretic complexity of certain collections of elements of \( \KK \) and of \( \MALG \) defined by means of the density function---these collections are natural and they deserve to be classified within the projective hierarchy.
It is convenient to introduce the following definition: if \( \mathcal{A} \subseteq \MEAS \) and \( \bGamma \) is a pointclass, then we say that \markdef{\( \mathcal{A} \) is \( \bGamma \) inside \( \KK \)} whenever \( \mathcal{A} \cap \KK \) is in \( \bGamma ( \KK ) \).

\begin{theorem}\label{th:hardness}
Let \( n < \omega \), and work inside \( \KK \).
\begin{enumerate-(a)}
\item\label{th:hardness-a}
The following collections of sets are \( \bPi\sp {1}\sb {1} \)-complete:
\begin{enumerate}[label={\upshape (a\arabic*)}, leftmargin=2pc]
\item\label{th:hardness-a-1}
\( \SHARP\sb { n } \), \( \SHARP\sb { \leq n } \), \( \SHARP\sb { < \omega } \), \( \SHARP\sb { \leq \omega } \), \( \SHARP\sb {\omega } \),
\item\label{th:hardness-a-2}
 \( \BLR\sb {n} \), \( \BLR\sb {\leq n} \), \( \BLR\sb { < \omega } \), \( \BLR\sb { \leq \omega } \), \(\BLR\sb { \omega} \),
\item\label{th:hardness-a-3}
 \( \Range \sb { \leq n } \), \( \Range \sb { < \omega } \), \( \Range \sb { \leq \omega } \), \( \Range\sb {\mathrm{inj}} \), \( \Range\sb {\MGR} \), \( \Range \sb { \lambda \leq a } \) for \( a \in \cointerval{ 0 }{1 } \) and \( \Range \sb { \lambda < a } \) for any \( a \in \ocinterval{ 0 }{1 }\),
\item\label{th:hardness-a-4}
\( \Solid \), \( \qDual \), and \( \Dual \),
\end{enumerate}
\item\label{th:hardness-b}
The following collections of sets are \( 2 \)-\( \bSigma\sp {1}\sb {1} \)-complete: \( \Range \sb { n + 1 } \), \( \Range \sb { \omega } \), \( \Range \sb { \lambda = a } \), and \( \Range ( S ) \) with \( a \in ( 0 ; 1 ) \) and \( \emptyset \neq S \subseteq ( 0 ; 1 ) \) countable.
\item\label{th:hardness-c}
\( \Spongy \) is \( 2 \)-\( \bSigma\sp {1}\sb {1} \) and it is both \( \bSigma\sp {1}\sb {1} \)-hard and \( \bPi\sp {1}\sb {1} \)-hard.
\end{enumerate-(a)}
\end{theorem}

\begin{proof}
With Theorem~\ref{th:upperbound} establishing the upper bounds, it is enough to focus on the hardness results.
We use the functions constructed in the preceding pages:
\begin{itemize}
\item
 \( \boldsymbol{E}\sb 2 \colon \PrTr\sb 2 \to \PrTr\sb 2 \) of~\eqref{eq:Explode}, and
\item
the reductions \( \firstreduction\sb c , \secondreduction , \thirdreduction\sb c \colon \PrTr\sb 2 \to \KK \) of Theorems~\ref{th:firstreduction}, \ref{th:secondreduction}, and~\ref{th:thirdreduction}.
\end{itemize}
Recall also that \( \Br{ n }{ 2 } \), the set of all pruned trees \( T \) on \( 2 \) with exactly \( n \) branches in \( \mathcal{N} \), is \( \bPi\sp {1}\sb {1} \)-complete by~\eqref{eq:Brarecomplete}, and so are the sets \( \Br{ \leq n }{ 2 }\), \( \Br{ < \omega}{ 2 }\), \( \Br{ \leq \omega}{ 2 } \).

We are now ready to prove the various clauses of the theorem.
For \( 0 < r < 1 \) let \( \mathbf{c} ( r ) \colon \mathcal{N} \to ( 0 ; 1 ) \) be the constant function with value \( r \). 
\begin{enumerate}[leftmargin=2pc]
\item[\ref{th:hardness-a-1}]
\( \firstreduction\sb{\mathbf{c} ( r )} ( T ) \in \Range\sb { \leq 1} \) for every \( T \in \PrTr \sb 2 \) and every \( r \in ( 0 ; 1 ) \), so
\begin{itemize}
\item
 \( T \in \WF\sb 2 \IFF \firstreduction\sb{\mathbf{c} ( r )} ( T ) \in \Range\sb {0} \), and hence \( \Range\sb {0} \) is \( \bPi\sp {1}\sb {1} \)-hard,
 \item
 \( T \in \WF\sb 2 \IFF \firstreduction\sb{\mathbf{c} ( r )} ( \boldsymbol{E}\sb 2 ( T ) ) \in \SHARP\sb { < \omega }\), so \( \SHARP\sb { < \omega } \) is \( \bPi\sp {1}\sb {1} \)-hard.
 \item
 \( T \in \Br{ \kappa }{2} \IFF \firstreduction\sb{\mathbf{c} ( r )} ( T ) \in \SHARP\sb { \kappa } \) and \( T \in \Br{ \leq \kappa }{2} \IFF \firstreduction\sb{\mathbf{c} ( r )} ( T ) \in \SHARP\sb { \leq \kappa } \), for all \( \kappa \leq \omega \), and hence \( \SHARP\sb { n } \), \( \SHARP\sb { \leq n } \), \( \SHARP\sb { \leq \omega } \) are \( \bPi\sp {1}\sb {1} \)-hard. 
\end{itemize}

Let \( U \in \Br{1}{2} \) and let 
\begin{equation}\label{eq:th:hardness-a}
 I \colon \PrTr\sb 2 \to\PrTr\sb 2 , \quad T \mapsto 0 \conc U \cup 1 \conc \boldsymbol{E}\sb 2 ( T ) .
\end{equation}
Then 
\[
\PrTr\sb 2 \to \KK , \quad T \mapsto \set{ 0\sp {( \omega ) }} \cup \textstyle \bigcup\sb {n \in \omega } 0\sp {( n )} \conc 1 \conc \firstreduction\sb{\mathbf{c} ( r )} ( I ( T ) ) 
\] 
maps \( \WF\sb 2 \) to \( \SHARP\sb { \omega} \) and \( \IF\sb 2 \) to \( \SHARP\sb { > \omega}\).
Therefore \( \SHARP\sb {\omega} \) is \( \bPi\sp {1}\sb {1} \)-hard.

\item[\ref{th:hardness-a-2}]
As \( T \in \Br{ n }{2} \iff \secondreduction ( T ) \in \BLR\sb { n } \) and \( T \in \Br{ \leq \kappa }{2} \iff \secondreduction ( T ) \in \BLR\sb { \leq \kappa } \) with \( \kappa \leq \omega \), then \( \BLR\sb {n} \), \( \BLR\sb { \leq n} \), and \( \BLR\sb { \leq \omega } \) are \( \bPi\sp {1}\sb {1} \)-hard.

The function \( \secondreduction \circ \boldsymbol{E}\sb 2 \) witnesses that \( \WF\sb 2 \leqW \BLR\sb {< \omega } \), and the function
\[ 
\PrTr\sb 2 \to \KK , \quad T \mapsto \set{ 0\sp {( \omega ) }} \cup \textstyle\bigcup\sb {n \in \omega } 0\sp {( n )} \conc 1 \conc \secondreduction ( I ( T ) ) 
\] 
with \( I \) as in~\eqref{eq:th:hardness-a} witnesses that \( \BLR\sb {\omega} \) is \( \bPi\sp {1}\sb {1} \)-hard.

\item[\ref{th:hardness-a-3}]
Let \( K \in \Range\sb {n} \cap \KK \), and let \( r \in ( 0 ; 1 ) \setminus \ran \density\sb K \).
Then 
\[
\PrTr\sb 2 \to \KK , \quad T \mapsto ( 0 \conc K ) \cup \left ( 1 \conc \firstreduction\sb{\mathbf{c} ( r )} ( T ) \right ) 
\]
witnesses that \( \WF\sb 2 \leqW \Range\sb { \leq n} \), so \( \Range\sb { \leq n} \) is \( \bPi\sp {1}\sb {1} \)-hard.

If the \( r\sb n \in (0 ; 1 ) \) are distinct, then
\[
 \PrTr\sb 2 \to \KK , \qquad T \mapsto \set{ 0\sp {( \omega ) }} \cup \textstyle\bigcup\sb {n \in \omega } 0\sp {( n + 1 )} \conc 1\sp {( n  )} \conc \firstreduction\sb { \mathbf{c} ( r\sb n ) } ( T )
\]
maps \( \WF\sb 2 \) to \( \Range\sb {< \omega } \) and \( \IF\sb 2 \) to \( \Range\sb \omega \), so \( \Range\sb {< \omega } \) is \( \bPi\sp {1}\sb {1} \)-hard and \( \Range\sb \omega \) is \( \bSigma\sp {1}\sb {1} \)-hard.

If \( K = \firstreduction\sb{ \mathbf{c} ( r ) } ( U ) \) for some fixed \( U \in \Br{1}{2} \), then
\[
\PrTr\sb 2 \to \KK , \quad T \mapsto \bigl ( 0 \conc K \bigr ) \cup \bigl ( 1 \conc \firstreduction\sb{ \mathbf{c} ( r ) } ( T) \bigr )
\]
witnesses that \( \WF\sb 2 \leqW \Range\sb { \mathrm{inj} } \).

If \( c \colon \mathcal{N} \to ( 0 ; 1 ) \) is continuous and injective, then \( T \in \Br{ \leq \omega }{2} \IFF \firstreduction\sb c ( T ) \in \Range\sb { \leq \omega } \), so \( \Range\sb {\leq \omega } \) is \( \bPi\sp {1}\sb {1} \)-hard.

If \( c \colon \mathcal{N} \to ( 0 ; 1 ) \) is continuous and surjective, then the map \( \thirdreduction\sb c \) of Theorem~\ref{th:thirdreduction} witnesses that \( \WF\sb 2 \leqW \Range\sb { \MGR } \), that \( \WF\sb 2 \leqW \Range\sb {\lambda \leq a } \) for any \( a \in \cointerval{ 0 }{ 1 } \), and that \( \WF\sb 2 \leqW \Range\sb { \lambda < a } \) for any \( a \in \ocinterval{ 0 }{ 1 } \).

\item[\ref{th:hardness-a-4}]
By~\eqref{eq:dual} \( \qDual = \SHARP\sb 0 = \Range\sb 0 \) so \( \qDual \) is \( \bPi\sp {1}\sb {1} \)-hard by part~\ref{th:hardness-a-1}.
As \( \secondreduction ( T ) \in \qDual \) for all \( T \in \PrTr\sb 2 \), and 
\[
T \in \WF\sb 2 \IMPLIES \secondreduction ( T ) \in \Dual \subseteq \Solid \quad\text{and} \quad T \in \IF\sb 2 \IMPLIES \secondreduction ( T ) \in \Spongy 
\]
therefore \( \Dual \) and \( \Solid \) are \( \bPi\sp {1}\sb {1} \)-hard.

\item[\ref{th:hardness-b}]
For \( 1 \leq n < \omega \) choose distinct \( r\sb i \in ( 0 ; 1 ) \) for \( i \leq n + 1\).
Let \( c , d \colon \pre{\omega}{ \omega } \to ( 0 ; 1 ) \) be continuous and such that \( \ran ( c ) = \set{ r\sb 0 , \dots , r\sb n } \) and \( \ran ( d ) = \ran ( c ) \cup \set{ r\sb { n + 1} } \).
The map
\[ 
\PrTr\sb 2 \times \PrTr\sb 2 \to \KK , \quad ( U , T ) \mapsto 0 \conc \thirdreduction\sb { d } ( U ) \cup 1 \conc \thirdreduction\sb { c } ( T ) 
\]
witnesses that \( \WF\sb 2 \times \IF\sb 2 \leqW \Range\sb {n + 1} \), so \( \Range\sb {n + 1} \) is \( 2 \)-\( \bSigma\sp {1}\sb {1} \)-hard.

Let \( c \colon \mathcal{N} \to ( 0 ; 1 ) \) be continuous and injective, and let \( r\sb n \in (0 ; 1 ) \) be distinct.
Then \( \PrTr\sb 2 \times \PrTr\sb 2 \to \KK \)
\[
 ( U , T ) \mapsto \set{ 0\sp {( \omega )} } \cup \bigl ( \textstyle\bigcup\sb {n \in \omega } 0\sp { ( n + 1 ) } \conc 1\sp { ( n  ) } \conc \firstreduction\sb {\mathbf{c} ( r \sb n )} ( T ) \bigr ) \cup 1 \conc \firstreduction\sb c ( \boldsymbol{E}\sb 2 ( U ) )
\] 
reduces \( \WF\sb 2 \times \IF\sb 2 \) to \( \Range\sb { \omega } \), so \( \Range\sb { \omega } \) is \( 2 \)-\( \bSigma\sp {1}\sb {1} \)-hard.

Let \( \emptyset \neq S \subseteq ( 0 ; 1 ) \) be countable, say \( S = \setof{ r\sb n }{ n \in \omega } \), and let \( r \in ( 0 ; 1 ) \setminus S \).
Then the map \( \PrTr\sb 2 \times \PrTr\sb 2 \to \KK \)
\[
( U , T ) \mapsto \set{ 0\sp {( \omega )} } \cup \bigl ( \textstyle \bigcup\sb { n \in \omega } 0\sp { ( n + 1 ) } \conc 1\sp { ( n ) } \conc \firstreduction\sb { \mathbf{c} ( r \sb n ) } ( T ) \bigr ) \cup \bigl ( 1 \conc \firstreduction\sb { \mathbf{c} ( r ) } ( U ) \bigr ) 
\]
witnesses that \( \WF\sb 2 \times \IF\sb 2 \leqW \Range ( S ) \).

Finally, for \( a \in ( 0 ; 1 ) \) let us show that \( \Range \sb { \lambda = a } \) is \( 2 \)-\( \bSigma\sp {1}\sb {1} \)-hard.
Let \( c\sb 1 , c\sb 2 \colon \pre{\omega}{\omega} \to ( 0 ; 1 ) \) be continuous and such that \( \ran c\sb 1 = ( b ; 1 ) \) and \( \ran c\sb 2 = ( 0 ; a ) \), where \( a \leq b < 1 \) and \( 1 - b \neq a \).
Then 
\[ 
\PrTr\sb 2 \times \PrTr\sb 2 \to \KK , \quad ( U , T ) \mapsto 0 \conc \thirdreduction\sb { c\sb 1 } ( U ) \cup 1 \conc \thirdreduction\sb { c\sb 2 } ( T ) 
\]
witnesses that \( \WF\sb 2 \times \IF\sb 2 \leqW \Range\sb { \lambda = a } \).

\item[\ref{th:hardness-c}] 
Theorem~\ref{th:firstreduction} shows that \( \WF\sb 2 \leqW \Spongy \) for any continuous function \( c \), and Theorem~\ref{th:secondreduction} shows that \( \IF\sb 2 \leqW \Spongy \).
\qedhere
\end{enumerate}
\end{proof}

\begin{question}
Is \( \Spongy \) \( 2 \)-\( \bSigma\sp {1}\sb {1} \)-complete?
\end{question}

The next result summarizes the content of parts~\ref{th:hardness-a-4} and~\ref{th:hardness-b} of Theorem~\ref{th:hardness}, and Theorem~\ref{th:uniformity}.

\begin{theorem}\label{th:range}
Let \( S \subseteq ( 0 ; 1 ) \) be \( \bSigma\sp {1}\sb {1} \).
Then \( \Range ( S ) \cap \KK \) is 
\begin{itemize}
\item
\( \bPi\sp {1}\sb {1} \)-complete, if \( S = \emptyset \),
\item
\( 2 \)-\( \bSigma\sp {1}\sb {1} \)-complete, if \( S \neq \emptyset \) is countable,
\item
\( \bPi\sp {1}\sb {2} \)-complete, if \( S \) is uncountable.
\end{itemize}
\end{theorem}

\begin{corollary}
The quasi-order \( \preceq \) on \( \KK \) defined by
\[
 K \preceq H \IFF \ran ( \density\sb {K} ) \subseteq \ran ( \density\sb {H} ) 
\]
is \( \varPi\sp 1\sb 2 \setminus \bSigma\sp {1}\sb {2} \).
Similarly the induced equivalence relation
\[
 K \sim H \IFF \ran ( \density\sb {K} ) = \ran ( \density\sb {H} ) 
\]
is \( \varPi\sp 1\sb 2 \setminus \bSigma\sp {1}\sb {2} \), and its equivalence classes are either \( \bPi\sp {1}\sb {2} \)-complete or \( 2 \)-\( \bSigma\sp {1}\sb {1} \)-complete, with the exception of a single class that is \( \bPi\sp {1}\sb {1} \)-complete.
\end{corollary}

\begin{proof}
 \( K\sb 1 \preceq K\sb 2 \IFF \FORALL{z\sb 1 } \left ( \oscillation\sb {K\sb 1} ( z\sb 1 ) = 0 \implies \EXISTS{z\sb 2 } \left ( \density\sb {K\sb 1} ( z\sb 1 ) = \density\sb {K\sb 2} ( z\sb 2 ) \right ) \right ) \), so \( \preceq \) is \( \varPi\sp {1}\sb {2} \). 
If \( K \) is such that \( \ran \density\sb K = [ 0 ; 1 ] \), then \( \setofLR{ H \in \KK }{ K \preceq H } \) is \( \bPi\sp {1}\sb {2} \)-complete, so \( \preceq \) is not \( \bSigma\sp {1}\sb {2} \).
The argument for \( \sim \) is analogous.
\end{proof}

If \( \mathcal{C} \subseteq \MEAS \) is a collection of sets such that \( \mathcal{C} \cap \KK \) is \( \bGamma \)-hard for certain projective pointclasses \( \bGamma \), then \( \widehat{\mathcal{C}} = \setof{ \eq{A} \in \MALG }{ A \in \mathcal{C} } \) is Borel-\( \bGamma \)-hard, since the map \( j \) of~\eqref{eq:embeddingKintoMALG} is Borel.

Therefore we have at once the following result.

\begin{corollary}\label{cor:hardness}
Let \( n < \omega \).
\begin{enumerate-(a)}
\item
The following subsets of \( \MALG \) are \( \bPi\sp {1}\sb {1} \)-complete:
\begin{itemize}
\item
 \( \widehat{\SHARP}\sb { n } \), \( \widehat{\SHARP}\sb { \leq n } \), \( \widehat{\SHARP}\sb { < \omega } \), \( \widehat{\SHARP}\sb { \leq \omega } \), \( \widehat{\SHARP}\sb { \omega } \),
\item
 \( \widehat{\BLR}\sb { n} \), \( \widehat{\BLR}\sb {\leq n} \), \( \widehat{\BLR}\sb { < \omega } \), \( \widehat{\BLR}\sb { \leq \omega } \), \( \widehat{\BLR}\sb { \omega } \),
\item
 \( \widehat{\Range} \sb { \leq n } \), \( \widehat{\Range} \sb { < \omega } \), \( \widehat{\Range }\sb { \leq \omega } \), \( \widehat{\Range}\sb {\mathrm{inj}} \), \( \widehat{\Range} \sb {\MGR} \), \( \widehat{\Range} \sb { \lambda \leq a } \) for any \( a \in \cointerval{ 0 }{ 1 } \) and \( \widehat{\Range} \sb { \lambda < a } \) for any \( a \in \ocinterval{ 0 }{ 1 } \),
\item
 \( \widehat{\Solid} \), \( \widehat{\qDual} \), and \( \widehat{\Dual} \).
\end{itemize}
\item
The following subsets of \( \MALG \) are Borel-\( 2 \)-\( \bSigma\sp {1}\sb {1} \)-complete:
\begin{itemize}
\item
\( \widehat{ \Range } \sb { n + 1 } \), \( \widehat{ \Range }\sb { \omega } \), \( \widehat{ \Range } ( S ) \) and \( \widehat{ \Range }\sb {\lambda = a } \) with \( a \in ( 0 ; 1 ) \) and \( S \neq \emptyset \) countable,
\item
\( \widehat{\Spongy} \) is \( 2 \)-\( \bSigma\sp {1}\sb {1} \) and it is both \( \bSigma\sp {1}\sb {1} \)-hard and \( \bPi\sp {1}\sb {1} \)-hard.
\end{itemize}
\item
\( \widehat{ \Range } ( S ) \) is Borel-\( \bPi\sp {1}\sb {2} \)-complete if \( S \subseteq ( 0 ; 1 ) \) is uncountable and analytic.
\end{enumerate-(a)}
\end{corollary}

Let \( X \) be an uncountable Polish space.
Recall that a pointclass \( \bGamma \) has the separation property if every pair of disjoint nonempty sets \( A, B \subseteq X \) in \( \bGamma \) can be separated by a set in \( \Delta ( \bGamma ) \equalsdef \bGamma \cap \dual{ \bGamma} \).
Assuming enough determinacy, for every \( \bGamma \) exactly one among \( \bGamma \) and \( \dual{\bGamma} \) has the separation property~\cite{Steel:1981lh}.
It can be shown in \( \ZFC \) that \( \bSigma\sp {1}\sb {1} \) and \( \bPi\sp {1}\sb {2} \) have the separation property.
The pointclass \( 2 \)-\( \bSigma\sp {1}\sb {1} \) does not have the separation property, since it has the pre-well-ordering property, as observed by John Steel (personal communication).

Nevertheless some of the \( \bPi\sp {1}\sb {1} \) and \( 2 \)-\( \bSigma\sp {1}\sb {1} \) sets considered in this paper \emph{can} be separated:
\begin{itemize}
\item
if \( 0 < a < b \leq 1 \) then \( \widehat{ \Range }\sb { \lambda = a } \subseteq \widehat{ \Range }\sb {\lambda \leq a } \) and \( \widehat{ \Range }\sb { \lambda \leq a } \cap \widehat{ \Range }\sb { \lambda = b } = \emptyset \), so by Theorem~\ref{th:upperbound} \( \widehat{ \Range }\sb { \lambda = a } , \widehat{ \Range }\sb { \lambda = b } \in 2 \text{-} \bSigma\sp {1}\sb {1} \) are separated by a set in \( \bPi\sp {1}\sb {1} \subset \Delta ( 2 \text{-} \bSigma\sp {1}\sb {1} ) \);
\item
suppose \( S\sb 1 , S\sb 2 \subseteq ( 0 ; 1 ) \) are analytic and distinct; without loss of generality we may assume that there is \( r \in S\sb 2 \setminus S\sb 1 \).
Let \( B \) be Borel and such that \( S\sb 1 \subseteq B \) and \( r \notin B \).
Then \( \widehat{ \Range } ( S\sb 1 ) \subseteq \widehat{ \Range } ( { \subseteq } B ) \) and \( \widehat{ \Range } ( { \subseteq } B ) \cap \widehat{ \Range } ( S\sb 2 ) = \emptyset \), and by Lemma~\ref{lem:rangecontainment} \( \widehat{ \Range } ( { \subseteq } B ) \in \bPi\sp {1}\sb {1} \).
\end{itemize}

A collection of sets \( \setofLR{A\sb i}{ i \in I } \) in \( \bGamma \) is said to be \markdef{\( \Delta ( \bGamma ) \)-inseparable} if they are pairwise disjoint and for every \( i \neq j \) there is no set in \( \Delta ( \bGamma ) \) that separates \( A\sb i \) from \( A\sb j \).
The collections \( \setof{ \widehat{\SHARP}\sb {n} }{ 1 \leq n < \omega } \) and \( \setof{ \widehat{\BLR}\sb {n} }{ 1 \leq n < \omega} \) are \( \bPi\sp {1}\sb {1} \)-inseparable in \( \MALG \).
In fact a stronger result holds.
For \( H \in \KK \), let
\begin{align*}
\SHARP\sp H = \setofLR{ K \in \KK }{ \Sharp ( K ) \backsimeq H } & , & \BLR\sp H = \setof{ K\in \KK }{ \Blur ( K ) \backsimeq H }
\end{align*}
where \( \backsimeq \) is the \markdef{homeomorphism relation}.

\begin{theorem} \label{th:SHARPcomplete}
Both \( \SHARP\sp H \) and \( \BLR\sp H \) are \( \bPi\sp 1\sb 1 \)-complete subsets of \( \KK \).
\end{theorem}

\begin{proof}
The homeomorphism classes in \( \KK \) are Borel, see \emph{e.g.}~\cite{Camerlo:2001qf}.
As \( \setof{ ( K , z ) \in \KK \times \pre{\omega}{2} }{ z \in \Sharp ( K ) } \) and \( \setof{ ( K , z ) \in \KK \times \pre{\omega}{2} }{ z \in \Blur ( K ) } \) are Borel, Theorem~\ref{th:sectionKechris} yields at once that \( \SHARP\sp H \) and \( \BLR\sp H \) are coanalytic.

For any \( U \) tree on \( \omega \) let
\[
U\sp \star = { \downarrow} \setof{ \check{u} }{ u \in U } \cup \setof{ \check{u} \conc 0\sp { (n) }}{ u \text{ is a terminal node of } U \text{ and } n\in\omega} .
\]
where \( \check{u} \) is as in~\eqref{eq:check(t)}.
For later use, notice that the map \( \Tr \to \PrTr\sb 2 \), \( U \mapsto U\sp \star \), is Borel.
Moreover for all \( U \in \Tr \), the map \( \boldsymbol{h} \restriction \body{ U } \) is a homeomorphism between \( \body{ U } \) and \( \body{U\sp \star } \cap \mathcal{N} \), where \( \boldsymbol{h}\) is as in~\eqref{eq:BaireinCantor0}.

Fix any continuous function \( c \colon \mathcal{N} \to ( 0 ; 1 ) \).
Since the map \( x \mapsto \overline{x} \) is injective and continuous, it is a homeomorphism onto its range, and the same is true for its restrictions.
Consequently, by Theorem~\ref{th:firstreduction} and using the function \( \firstreduction\sb c \) defined there,
\[
 \body{U} \backsimeq \body{ U\sp \star } \cap \mathcal{N} \backsimeq \Sharp ( \firstreduction\sb c ( U \sp \star ) )
\]
for every \( U \in \Tr \).
In particular, choosing \( U \) such that \( \body{U} \backsimeq H \), one has \( \firstreduction\sb c ( U\sp {\star } ) \in \SHARP\sp H \).
The map \( f \colon \PrTr \sb 2 \to \KK \) 
\[ 
f ( T ) = \bigcup\sb{n\in\omega } 0\sp{ ( n + 1 ) } \conc 1\sp{ ( n ) } \conc \firstreduction\sb c ( T ) \cup 1 \conc \firstreduction\sb{c} ( U\sp{\star } )
\]
is continuous.
Moreover, if \( T\in \WF\sb 2 \) then \( f ( T ) \in \SHARP\sp H \), while if \( T\in \IF\sb 2 \) then \( \Sharp ( f ( T ) ) \) is not compact.
This shows that \( \SHARP\sp H \) is \( \bPi\sp 1\sb 1 \)-complete.

For \( \BLR\sp H \) employ a similar argument, using \( \secondreduction \) instead of \( \firstreduction\sb c \).
\end{proof}

We can now give an example of a large collection of Borel-inseparable, complete coanalytic subsets of \( \KK \). 

\begin{theorem}\label{th:Pi11inseparable}
Let \( H , H' \in \KK \), with \( H \not\backsimeq H' \).
Then:
\begin{enumerate-(a)}
\item\label{th:Pi11inseparable-a}
the sets \( \SHARP\sp H, \SHARP\sp {H'} \) are disjoint and Borel-inseparable
\item\label{th:Pi11inseparable-b}
 the sets \( \BLR\sp H, \BLR\sp {H'} \) are disjoint and Borel-inseparable
\end{enumerate-(a)}
\end{theorem}

\begin{proof}
Recall the Borel map \( \Tr \to \PrTr\sb 2 \), \( U \mapsto U\sp \star \) from the proof of Theorem~\ref{th:SHARPcomplete} and notice that the Borel function \( \Tr \to \KK \), \( U \mapsto \firstreduction\sb c ( U\sp \star ) \) reduces \( \setof{ U\in \Tr }{ \body{U} \backsimeq H } \) to \( \SHARP\sp H \), for any \( H \in \KK \).
Given \( H \not\backsimeq H' \), since \( \setofLR{ U\in \Tr }{ \body{U} \backsimeq H} \), \( \setofLR{ U\in \Tr }{ \body{U} \backsimeq H' } \) are complete coanalytic and Borel-inseparable by~\cite[Theorem 1.4]{Camerlo:2002sf}, the same holds for \( \SHARP\sp H , \SHARP\sp {H'} \).
This yields~\ref{th:Pi11inseparable-a}.

For~\ref{th:Pi11inseparable-b}, employ a similar argument, using the function \( \secondreduction \) from Theorem~\ref{th:secondreduction} instead of \( \firstreduction\sb c \).
\end{proof}

\printbibliography

\end{document}